\documentclass[journal]{IEEEtran}

\usepackage[T1]{fontenc}
\usepackage{microtype}     
\usepackage{amsmath}
\usepackage{amssymb}
\usepackage{mathtools}
\usepackage{bm}            
\usepackage{bbm}
\usepackage{nicefrac}     

\usepackage{graphicx}
\usepackage[dvipsnames]{xcolor}
\usepackage{tikz}
\usepackage{pgfplots}
\usepgfplotslibrary{groupplots}
\pgfplotsset{compat=1.18}  
\usetikzlibrary{
    arrows.meta,
    automata,
    decorations.pathreplacing,
    patterns,
    patterns.meta,
    positioning,
    shapes.geometric,
    calc,               
    matrix
}
\usepackage[caption=false,font=footnotesize]{subfig}

\usepackage{booktabs}       
\usepackage{multirow}
\usepackage{array}     
\usepackage{enumitem}
\usepackage{amsthm}

\usepackage{url}
\usepackage{cite}              
\usepackage{orcidlink}
\usepackage{hyperref}

\theoremstyle{plain}
\newtheorem{theorem}{Theorem}
\newtheorem{lemma}[theorem]{Lemma}

\theoremstyle{definition}
\newtheorem{definition}{Definition}

\theoremstyle{remark}
\newtheorem{remark}{Remark}

\definecolor{ultramarineblue}{rgb}{0.05, 0.6, 0.96}
\colorlet{mycolor1}{teal!20!white}
\colorlet{mycolor2}{ultramarineblue!20!white}

\usepackage[table]{xcolor}
\definecolor{HeaderBlue}{RGB}{31,78,121}
\definecolor{SubHeaderBlue}{RGB}{221,235,247}
\definecolor{MetricGray}{RGB}{242,242,242}

\newcommand{\nmax}{n_{\mathrm{max}}}
\newcommand{\transpose}{\mathrm{T}}

\newcommand{\CP}{\mathrm{C_P}}  
\newcommand{\CS}{\mathrm{C_S}}

\newcommand{\BP}{B_{\CP}}
\newcommand{\BS}{B_{\CS}}
\newcommand{\LBP}{\widetilde{B}_{\CP}}
\newcommand{\LBS}{\widetilde{B}_{\CS}}

\newcommand{\E}{\mathbb{E}}                 
\newcommand{\var}{\mathbb{V}\mathrm{ar}}    
    
\newcommand{\Prob}{\mathbb{P}}              
\newcommand{\diff}{\,\mathrm{d}}              

\DeclareMathOperator{\RealPart}{Re}
\DeclareMathOperator{\Det}{det}
\DeclareMathOperator{\adj}{adj}
\DeclareMathOperator{\Exp}{Exp}

\begin{document}

% --- Title and Author Information (using \thanks for affiliations) ---
\title{Analysis of Triggered Packet Streams: A Matrix-Analytic Method for Exponential Triggering Delays}

\author{
    Mehran Rahnamania\orcidlink{0009-0003-9098-5586},
    Michel Mandjes\orcidlink{0000-0001-6783-4833},
    and Farid Ashtiani\orcidlink{0000-0002-6955-1711}
    \thanks{Mehran Rahnamania and Farid Ashtiani are with the Department of Electrical Engineering, Sharif University of Technology, Tehran, Iran (e-mails: mehran.r109@yahoo.com; ashtianimt@sharif.edu).}
    \thanks{Michel Mandjes is with the Mathematical Institute, Leiden University, P.O. Box 9512, 2300 RA Leiden, The Netherlands, and also with the Korteweg-de Vries Institute for Mathematics, University of Amsterdam, Amsterdam, The Netherlands (e-mail: m.r.h.mandjes@math.leidenuniv.nl).}
    \thanks{Manuscript received Month Day, Year; revised Month Day, Year. This work was supported by ...}
}
\maketitle

% --- Abstract and Keywords ---
\begin{abstract}
In many communication networks, the transmission of a packet may automatically trigger the transmission of a subsequent packet from the same source after a (possibly random) delay, without requiring acknowledgment or feedback. Such behavior arises in multi-stage status updating, proactive protocols, and other applications where users generate causally dependent packet streams. In this paper, in order to analyze these systems, we introduce the $\mathrm{M^T/G/1}$ queue. In this model, primary customers arrive according to a Poisson process, and each primary customer triggers a secondary customer to join the queue after an independent delay. This arrival mechanism falls outside the scope of classical queueing models with renewal arrival processes. When the triggering delays follow an exponential distribution, we exploit the memoryless property to set up a tractable Markov description. 
By truncating the number of pending secondary customers, we derive a finite system of linear algebraic equations in the Laplace--Stieltjes transform domain and solve them using matrix-analytic methods. Based on the resulting workload distribution, we compute class-specific performance metrics using PASTA for primary customers and Palm conditioning for secondary customers. Finally, we validate the accuracy of this truncation through numerical experiments.
\end{abstract}


% --- Main Body Sections ---
\section{Introduction}
\label{sec:intro}

\IEEEPARstart{I}{n} modern communication networks, the transmission of a packet may automatically trigger a subsequent packet to be transmitted from the same source after a (possibly random) delay, without requiring acknowledgment or feedback on the status of the primary transmission~\cite{11571174,Rahnamania2024CorrelatedArrivals}. This \textit{triggering} dynamic underlies applications ranging from the Internet of Things (IoT) to next-generation wireless systems, where numerous independent users operate concurrently, sharing a common radio resource, yet each user exhibits an internal causal dependency across its own packet transmissions.

A representative example arises in industrial and environmental IoT monitoring. When an abnormal condition is detected (e.g., a fire or gas leak), a sensor node immediately transmits a short alarm message to a gateway. In parallel, it begins collecting detailed environmental data (e.g., images, waveforms, or high-resolution telemetry), which is transmitted only after hardware- and processing-induced delays~\cite{s23073544}. A similar pattern occurs in wireless systems designed for ultra-reliable low-latency communication (URLLC), where, in a basic setting, a device transmits a packet and automatically schedules a replica transmission after a deterministic or random delay, depending on the system design, to improve reliability without the additional latency overhead introduced by traditional feedback mechanisms~\cite{9174916}.

These examples illustrate a common communication pattern in which many independent users operate concurrently on a shared communication link, while each user’s packet transmissions are causally dependent. Despite its practical relevance, the rigorous performance analysis of such systems --- typically carried out using queueing-theoretic methods --- remains largely underexplored.

\subsection{Modeling and Analysis Challenges}
Traditional queueing theory relies heavily on the assumption of renewal arrival processes~\cite{Kleinrock1975Volume1}. While extensions such as Markov-modulated arrivals (MMPP, BMAP)~\cite{FISCHER1993149,chakravarthy2001batch,Latouche1999MatrixAnalytic}, Hawkes processes~\cite{10.1093/biomet/58.1.83}, and long-range dependent models~\cite{282603} capture correlations or clustering in aggregate traffic, they are limited in representing micro-level causal dependencies between individual arrivals. Exact queueing analysis under these models is typically restricted to a few special cases or relies on heavy-traffic approximations~\cite{4863c94f-0a3e-32f9-8292-01e3902d6958,daw2018queues}. \textit{G-networks} provide a queueing framework for systems with internal interactions that induce dependencies within arrivals~\cite{6796323,ec60910f-c7c5-39b2-8655-659b6b05ba84}. Positive and negative signals are used to model concurrent movements~\cite{1130000795627348864}. However, maintaining quasi-reversibility (QR) --- a sufficient condition for product-form solutions --- requires certain structural conditions~\cite[Ch. 3 and 4]{1130000795627348864} that are not satisfied by the causal dependencies inherent in our model.

By contrast, the systems we consider exhibit \emph{triggered arrivals}: each independent source transmits a \textit{primary} packet that, after a random delay, triggers a \textit{secondary} packet. When the transmissions from all sources are aggregated at a gateway, the resulting arrival stream exhibits dependencies that are not adequately captured by standard models in the literature. Modeling and analyzing such triggered-arrival systems, therefore, requires different methodological approaches.

The resulting dependency structure introduces several analytical challenges. Classical queueing results, such as those for $\mathrm{M/G/1}$ or $\mathrm{G/M/1}$ systems, rely on independent interarrival processes and are therefore inapplicable. Models based on correlated arrivals, such as MAPs, struggle to represent the triggering mechanism, since secondary arrivals may occur long after their corresponding primaries and can be interleaved with many independent customers from other sources. Self-exciting processes, such as Hawkes models, capture clustering effects but do not reflect the explicit one-to-one pairing and delay structure that characterize triggered arrivals. Finally, matrix-analytic methods become intractable, as an exact Markovian description would need to track both the queue length and the residual triggering delay of each pending secondary arrival; in an infinite-buffer system, this leads to an unbounded, infinite-dimensional state space.

To address this gap, we introduced the $\mathrm{M^T/M/1}$ queue in our preliminary work~\cite{Rahnamania2024CorrelatedArrivals}, in which \textit{primary} customers arrive according to a Poisson process with rate $\lambda$, and each arrival triggers a \textit{secondary} customer after an independent, generally distributed \textit{triggering delay}. The Poisson assumption for primary arrivals is justified by the fact that the superposition of many independent point processes (e.g., from {a large number} of sensors or users) is well approximated by a Poisson process.

An initial heuristic approximation in~\cite{Rahnamania2024CorrelatedArrivals} mapped the triggered-arrival system onto a dual quasi-reversible tandem queue augmented with redundant customers. Although this approximation provides estimates for mean system times, it does not yield higher-order moments or full distributional information.
Subsequently, in~\cite{11571174}, we refined the analysis by introducing an auxiliary customer class and a Markov-chain coupling argument. This allowed us to show that a secondary arrival sees the system either in steady state or with exactly one additional customer. Based on this structural property, we developed a second heuristic approximation that represents the system as a multi-class $\mathrm{M/G/1}$ queue with effective service times, enabling the derivation of an approximate Laplace--Stieltjes transform (LST) of the system-time distribution for each customer class. Both~\cite{Rahnamania2024CorrelatedArrivals} and~\cite{11571174}, however, are strictly restricted to \textit{exponential} service times that must be \textit{identical} for both primary and secondary customers, and their derivations rely on heuristic arguments. 

To overcome these limitations, this paper establishes a framework for the $\mathrm{M^T/G/1}$ queue, in which the service times of primary and secondary customers are allowed to \textit{differ} and follow \textit{general} distributions, while the triggering delays follow an exponential distribution. We exploit the memoryless property of the triggering mechanism to construct a tractable two-dimensional Markovian description of the system state, characterized by the virtual queue workload and the number of pending secondary customers. Using Laplace--Stieltjes transforms (LSTs) and a finite-state truncation scheme with appropriate boundary conditions, we reduce the steady-state integro-differential equations to a compact matrix-algebraic system, which we solve using matrix-analytic methods. Finally, we characterize the performance of both customer classes: waiting times for primary customers are derived using the Poisson arrivals see time averages (PASTA) property, while Palm conditioning is used to describe the state seen by the secondary customers. This framework yields distributions, as well as moments of system times such as means and variances, to any desired accuracy.

\subsection{Contributions}
The key contributions of this paper are as follows:
\begin{itemize}
    \item \textit{Structural characterization of triggered arrivals}: We provide a formal mathematical proof (Appendix~\ref{app:proof_arrival}) demonstrating that the aggregate arrival process in the $\mathrm{M^T/G/1}$ queue is not a renewal process for any non-degenerate triggering delay distribution. This specific non-renewal structure imposes fundamental limitations on classical queueing models that assume renewal-based or correlated arrival processes.
    
    \item \textit{Matrix-analytic framework for generally distributed service times}: Overcoming the heuristic limitations of prior models, we develop a mathematical formulation that accommodates general, class-distinct service time distributions. Exploiting the memoryless property of the exponential triggering delays, we construct a two-dimensional Markovian representation of the system state, formulate a finite-state approximation, and obtain a matrix-analytic solution for the stationary {workload} distribution (to any desired accuracy).

    \item \textit{Class-specific performance metrics}: We identify the structural differences in how primary and secondary arrivals interact with the queue. Using the PASTA property for primary customers and Palm conditioning for secondary customers, we obtain the LST representations of the distributions, as well as closed-form expressions for the means and variances of their waiting and system times.
\end{itemize}

\subsection{Organization}
The remainder of this paper is organized as follows.
Section~\ref{sec:app} motivates the theoretical framework by examining specific practical applications, followed by the formal definition of the $\mathrm{M^T/G/1}$ system in Section~\ref{sec:model}. In Section~\ref{sec:exp}, a matrix-analytic solution is presented for triggering delays following an exponential distribution and class-specific generally distributed service times. Section~\ref{sec:numerical} provides comprehensive numerical results and validation simulations alongside a detailed structural discussion. Finally, Section~\ref{sec:conclusion} provides concluding remarks.

\section{Applications in Communication}
\label{sec:app}

The $\mathrm{M^T/G/1}$ queue captures a proactive, open-loop communication dynamic common to modern systems in which the transmission of a \emph{primary} packet may automatically trigger the transmission of a corresponding \emph{secondary} packet from the same source after a (possibly random) delay. {This mechanism differs from those used in traditional reactive, feedback-driven protocols (e.g., ARQ).} In this section, we examine applications in which the triggered-arrival behavior is inherent to the system design and demonstrate why the $\mathrm{M^T/G/1}$ framework is a fundamental tool for modeling and analyzing many communication systems.

\subsection{Multi-Stage Status Updates and Age of Information}
In sensing and control architectures, data acquisition is often a multi-stage process. An IoT anomaly detection system may transmit an immediate, lightweight alarm (the \emph{primary} packet) while simultaneously initiating the collection of a high-fidelity environmental report. This data-intensive report is transmitted as the \emph{secondary} packet only after a hardware- or processing-induced delay~\cite{7466086,9380899}. 

The resulting two-stage update paradigm is directly relevant to the analysis of age of information (AoI), as the freshness of information depends not only on update frequency but also on the internal structure of individual update cycles. However, existing AoI queueing models~\cite{YatesKaul2019,Moltafet2020,10898038} typically assume mutually independent update streams and therefore cannot capture the micro-level causal dependency between a primary status update and its corresponding delayed update.

\subsection{Proactive Wireless Protocols and URLLC}
To meet the stringent reliability and latency requirements of wireless communications (e.g., URLLC in 5G and 6G), protocols are shifting from reactive to proactive schemes. For example, a device sends an initial transmission (the \emph{primary} packet) and automatically schedules a proactive replica (the \emph{secondary} packet) to be sent after a time offset. This ``blind repetition'' strategy occurs regardless of the first packet's success, bypassing the latency penalty of feedback cycles~\cite{9174916}. 

\subsection{Split Inference for Edge AI}
In Edge Intelligence, a device running a deep neural network (DNN) often employs an ``early exit'' architecture. When an input is sampled, the device performs a rapid, shallow inference and immediately transmits the preliminary result (the \emph{primary} packet) to the edge server. In parallel, the device continues to compute the deeper layers of the DNN to extract a refined feature vector. This computation introduces a processing delay. Upon completion, the feature vector (the \emph{secondary} packet) is queued for transmission over the same wireless interface~\cite{10621887,7488250}.

\vspace{0.75em}

These examples demonstrate that triggering dynamics are not merely theoretical curiosities, but arise naturally in the operation of many real-world communication and sensing systems.
To unify these application domains, we now formulate the triggering dynamics within a rigorous queueing-theoretic framework.

\section{System Model and Challenges}
\label{sec:model}
Motivated by the presence of triggering dynamics in communication networks, we now introduce a queueing-theoretic framework for their analysis. In this section, we formally define the $\mathrm{M^T/G/1}$ queue, describe the probabilistic mechanisms governing primary and secondary arrivals, and highlight the structural properties of the aggregate arrival process that violate the classical renewal assumption.

\subsection{Arrival Process and Queue Dynamics}
\label{subsec:model:description}

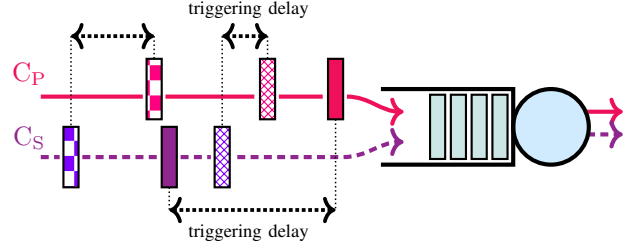
\begin{figure}
    \centering
    \begin{tikzpicture}[scale=1]
    
    \foreach \x in { 0.65, 0.925, 1.2,1.475}
    \draw[thick,black,fill=mycolor1] (\x+0.5, 0.075) rectangle (\x+0.7, 0.925);

    \draw[->, line width=0.55mm, OrangeRed] (2.7, 0.7) -- (3.7, 0.7);
    \draw[->, line width=0.55mm,Plum, densely dashed] (2.7, 0.4) -- (3.7, 0.4);
    \node[Plum] at (-4.15,0.5-0.15) {$\CS$};
    \node[OrangeRed] at (-4.15,1.2) {$\CP$};
    \node[black] at (-1.25,2.15-0.1) {\scriptsize triggering delay};
    \node[black] at (-1.25,-0.9) {\scriptsize triggering delay};

    \draw[thick,black,fill=OrangeRed]  (-0.2,0.8-0.2) rectangle (0,1.6-0.2);
    \draw[thick,black,fill=Plum]  (-0.2-2.2,0.1-0.4) rectangle (0-2.2,0.8-0.3);

    \draw[thick,black,pattern=crosshatch, pattern color=OrangeRed]  (-0.2-0.9,0.8-0.2) rectangle (0-0.9,1.6-0.2);
    \draw[thick,black,pattern=crosshatch, pattern color=Plum]  (-0.2-1.5,0.1-0.4) rectangle (0-1.5,0.8-0.3);

    \draw[thick,black, pattern=checkerboard, pattern color=OrangeRed]  (-2.6,0.6) rectangle (-2.4,1.4);
    \draw[thick,black, pattern=checkerboard, pattern color=Plum]  (-3.7,0.1-0.4) rectangle (-3.5,0.8-0.3);
    
    \draw[->,  line width=0.55mm, OrangeRed] (0.05, 1.1-0.2)  to[in=180, out=0] (0.8, 0.7);
    \draw[-,  line width=0.55mm, OrangeRed] (-0.85, 1.1-0.2) -- (-0.27, 1.1-0.2);
    \draw[-,  line width=0.55mm, OrangeRed] (-2.35, 1.1-0.2) -- (-1.175, 1.1-0.2);
    \draw[-,  line width=0.55mm, OrangeRed] (-4, 1.1-0.2) -- (-2.65, 1.1-0.2);

    \draw[densely dashed,->,  line width=0.55mm, Plum] (-1.4, 0.4-0.3) -- (0, 0.1)--(0, 0.4-0.3) to[in=180, out=0] (0.8, 0.3);
    \draw[densely dashed,-,  line width=0.55mm, Plum] (-2.15, 0.4-0.3) -- (-1.75, 0.4-0.3);
    \draw[densely dashed,-,  line width=0.55mm, Plum] (-3.45, 0.4-0.3) -- (-2.45, 0.4-0.3);
    \draw[densely dashed,-,  line width=0.55mm, Plum] (-4, 0.4-0.3) -- (-3.75, 0.4-0.3);

      \draw[-, line width=0.2mm, densely dotted, black] (-0.1, 0.6)  -- (-0.1, -0.6);
    \draw[<->, line width=0.5mm, densely dotted, black] (-2.3, -0.6)  -- (-0.1, -0.6);
    \draw[-, line width=0.2mm, densely dotted, black] (-2.3, -0.3)  -- (-2.3, -0.6);

    \draw[<->, line width=0.5mm, densely dotted, black] (-1.6, 1.9-0.2)  -- (-1, 1.9-0.2);
    \draw[<->, line width=0.5mm, densely dotted, black] (-3.6, 1.9-0.2)  -- (-2.5, 1.9-0.2);
    \draw[-, line width=0.2mm, densely dotted, black] (-3.6, 1.7)  -- (-3.6, 0.5);
    \draw[-, line width=0.2mm, densely dotted, black] (-2.5, 1.7)  -- (-2.5, 1.4);
    
    \draw[-, line width=0.2mm, densely dotted, black] (-1.6, 1.9-0.2)  -- (-1.6, 0.5);
    \draw[-, line width=0.2mm, densely dotted, black] (-1, 1.9-0.2)  -- (-1, 1.5-0.1);

    \draw[line width=0.55mm,fill=mycolor2]  (2.75,0.5) circle [radius=0.5];
    \node at (2.75,0.5) {};
    \draw[line width=0.55mm] (0.5,0) -- (2.25,0) -- (2.25,1) -- (0.5, 1);
    
\end{tikzpicture} 
    \caption{Illustration of the $\mathrm{M^T/G/1}$ queue. Primary customers ($\CP$) arrive via a Poisson process, and each triggers a corresponding secondary customer ($\CS$) to join the queue after a (possibly random) delay, referred to as the triggering delay. In this paper, triggering delays follow an exponential distribution.}
    \label{fig:model:system}
\end{figure}

We consider a single-server queueing system with an infinite buffer capacity, operating under the first-come-first-served (FCFS) service discipline. 
We denote this system by $\mathrm{M^T/G/1}$, where ``T'' signifies the specific triggering structure detailed below.
As illustrated in Fig.~\ref{fig:model:system}, the system is defined by the interactions between two distinct classes of customers:
\begin{itemize}
    \item \emph{Primary Customers} ($\CP$): These customers arrive according to a homogeneous Poisson process with rate~$\lambda$.
    Their service times $X_{\CP,1}, X_{\CP,2}, \dots$ are i.i.d., denoted by a generic random variable $X_{\CP}$, following a general distribution function $\BP(x) \coloneqq \Prob (X_{\CP} \le x)$ with LST $\LBP(s) \coloneqq \E[e^{-sX_{\CP}}]$.
    
    \item \emph{Secondary Customers} ($\CS$): Each arrival of a $\CP$ triggers a corresponding $\CS$ to join the queue after a random delay, referred to as the \emph{triggering delay}. These triggering delays $\{Y_i\}_{i\ge1}$ are independent of the service times of both classes and are modeled as i.i.d.\ random variables, represented by a generic non-negative random variable $Y$ with a general distribution function $F_Y(y)\coloneqq \Prob (Y \le y)$.
    Their service times $X_{\CS,1}, X_{\CS,2}, \dots$ are i.i.d., denoted by a generic random variable $X_{\CS}$, following a general distribution function $\BS(x) \coloneqq \Prob (X_{\CS} \le x)$ with LST $\LBS(s) \coloneqq \E[e^{-sX_{\CS}}]$.
\end{itemize}

The total average arrival rate to the system is $ 2\lambda$, as each $\CP$ generates exactly one $\CS$. For queueing stability, the total arrival rate must be less than the service rate, i.e.,
\begin{equation}\label{eq:model:stability}
    \rho = \lambda \big( \E[X_{\CS}] + \E[X_{\CP}] \big) < 1,
\end{equation}
where $\rho$ is the system utilization factor~\cite{asmussen2003applied}. Our analysis assumes the system operates in this stable regime. 

To understand the fundamental difficulty of analyzing this system, we must first establish the statistical properties of the arrival process.

\begin{lemma}\label{lem:model:arrival_props}
  The arrival process of an $\mathrm{M^T/G/1}$ queue with a non-degenerate triggering delay distribution exhibits the following properties:
    \begin{enumerate}[label=(\roman*)]
        \item The marginal arrival process of $\CP$ customers is a Poisson process.
        \item The marginal arrival process of $\CS$ customers is a Poisson process.
        \item The aggregate arrival process of $\CP$ and $\CS$ customers is \emph{not} a renewal process.
    \end{enumerate}
\end{lemma}

\begin{proof}
    See Appendix~\ref{app:proof_arrival}.
\end{proof}
We note that the proof of this lemma is not immediate due to the dependence between the marginal streams and therefore requires an explicit argument. Indeed, there exist constructions of two dependent renewal processes whose superposition remains a renewal process~\cite{Jacod1975DependentPoisson}.

Lemma~\ref{lem:model:arrival_props} implies that although the marginal arrival processes of both $\CP$ and $\CS$ customers are individually Poisson (each with rate~$\lambda$), the \emph{aggregate arrival process} is \emph{not} a Poisson process. 
As established in the lemma, this arises from the inherent dependency between the two streams: each $\CS$ arrival occurs at the arrival time of its corresponding $\CP$ plus the triggering delay~$Y$. 
Consequently, the aggregate process does not, in general, constitute a renewal process.
Hence, the \emph{first analytical challenge} is that this interdependence prevents the direct application of standard queueing results, such as those based on embedded Markov chains for $\mathrm{M/G/1}$ or $\mathrm{G/M/1}$ systems.

The \emph{second analytical challenge} stems from the fact that this dependency structure differs fundamentally from that of established models such as Batch Markovian Arrival Processes (BMAPs) or Markov-Modulated Poisson Processes (MMPPs). 
While such models capture correlations between \emph{successive} interarrival times in the \emph{aggregate} stream, our model directly links each arrival epoch of a $\CP$ and a $\CS$ via a triggering mechanism, where a $\CS$ may arrive long after its corresponding $\CP$ and be interleaved with many other customers.

The \emph{third analytical challenge} concerns the dimensionality of the state space. A direct application of matrix-analytic methods becomes intractable, since a complete state description must track not only the number of customers in the queue, but also the residual triggering delay associated with every pending $\CS$ arrival. In an infinite-buffer system, the number of such pending arrivals is itself unbounded, resulting in an infinite-dimensional state space.

Even when approximating this state space by a finite one --- for instance, by truncating the number of pending $\CS$ customers --- the model still suffers from the \emph{curse of dimensionality}. Tracking multiple concurrent residual triggering delays leads to exponential growth of the state space with respect to the number of pending $\CS$ customers, rendering the approach computationally infeasible. Moreover, constructing the corresponding transition matrix becomes highly intricate, as it must simultaneously capture: (i) Poisson arrivals of $\CP$, (ii) the initiation of new triggering delays, (iii) the residual lifetimes of all active triggering delays, (iv) the completion of triggering delays resulting in $\CS$ arrivals, and (v) the remaining service time of the customer currently in service.

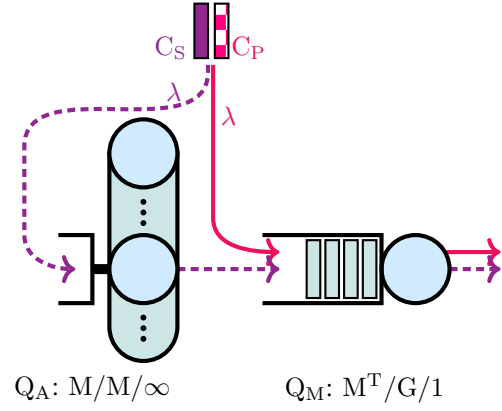
\begin{figure}
    \centering
        \begin{tikzpicture}[scale=0.9]
        \draw[->, line width=0.55mm, OrangeRed] (-0.725, 3.5) -- (-0.725, 1.25) to[in=180, out=-90] (0.25, 0.75);
        \draw[->, line width=0.55mm, OrangeRed] (2.6, 0.75) -- (3.5, 0.75);

        \draw[-, very thick, black] (-2.5, 0.5) -- (-2.25, 0.5);
        \draw[-, very thick, black] (-2.5, 0.525) -- (-2.25, 0.525);
        \draw[-, very thick, black] (-2.5, 0.55) -- (-2.25, 0.55);
        \draw[-, very thick, black] (-2.5, 0.475) -- (-2.25, 0.475);
        \draw[-, very thick, black] (-2.5, 0.45) -- (-2.25, 0.45);

        \draw[densely dashed,->, line width=0.55mm, Plum] (-0.8, 3.5) -- (-0.8, 3.4)  to[in=90, out=-90] (-3.5,2.3) -- (-3.5, 1)to[in=180, out=-90] (-2.75, 0.5);

        \draw[densely dashed,->, line width=0.55mm, Plum] (-1.25, 0.5) -- (0.25, 0.5);

        \draw[densely dashed,->, line width=0.55mm, Plum] (2.75, 0.5) -- (3.5, 0.5);

        \foreach \x in {0.65, 0.925, 1.2,1.475}
        \draw[thick,black,fill=mycolor1] (\x, 0.075) rectangle (\x+0.2, 0.925);

        \draw[ thick,pattern=checkerboard, pattern color=OrangeRed] (-0.7,3.6) rectangle (-0.5,4.4) ;
        \draw[ thick, black,fill=Plum] (-0.8,3.6) rectangle (-1,4.4) ;

        \node[OrangeRed] at (-0.2,3.75){$\CP$};
        \node[OrangeRed] at (-0.5,2.75){$\lambda$};
        \node[Plum] at (-1.35,3.75){$\CS$};
        \node[Plum] at (-1.3,3.1){$\lambda$};

        \draw[line width=0.55mm,fill=mycolor2]  (2.25,0.5) circle [radius=0.5];

        \draw[line width=0.55mm,fill=mycolor1] (-2.25,-0.35) arc[start angle=180, end angle=360,radius=0.5cm] --(-1.25,-0.35) -- (-1.25,2.2) -- (-2.25, 2.2) -- (-2.25,-0.35);

        \draw[line width=0.55mm,fill=mycolor2]  (-1.75,2.2) circle [radius=0.5];

        \draw[line width=0.55mm,fill=mycolor2]  (-1.75,0.5) circle [radius=0.5];

        \node at (-1.75,-0.25) {\huge\vdots};
        \node at (-1.75, 1.45) {\huge\vdots};

        \node at (1.5, -1.25) {$\mathrm{Q_M}$: $\mathrm{M^T/G/1}$};
        \node at (-2.5, -1.25) {$\mathrm{Q_A}$: $\mathrm{M/M/\infty}$};
        \draw[line width=0.55mm] (-3,0) -- (-2.5,0) -- (-2.5,1) -- (-3, 1);
        \draw[line width=0.55mm] (0,0) -- (1.75,0) -- (1.75,1) -- (0, 1);

    \end{tikzpicture}
    \caption{This figure depicts the auxiliary queue ($\mathrm{Q_A}$) that models the triggering delays between $\CP$ and $\CS$ customers. The main queue ($\mathrm{Q_M}$) operates in tandem with $\mathrm{Q_A}$; upon each $\CP$ arrival at $\mathrm{Q_M}$, a corresponding $\CS$ simultaneously arrives at $\mathrm{Q_A}$.}
    \label{fig:model:aux}
\end{figure}

\section{Matrix-Analytic Solution}
\label{sec:exp}

In this section, we develop a matrix-analytic solution for the case in which service times follow class-specific general distributions, while the triggering delays follow an exponential distribution, i.e., $Y \sim \Exp(\gamma)$. We first derive the stationary distribution of the system workload. We then characterize the system times experienced by primary customers ($\CP$) and secondary customers ($\CS$). Finally, numerical experiments demonstrate the accuracy of the analytical framework.

\subsection{The Auxiliary Queue}
\label{subsec:exp:aux_queue}

To mathematically track the pending $\CS$ customers, as illustrated in Fig.~\ref{fig:model:aux}, it is useful to introduce a conceptual queue, which we refer to as the \textit{auxiliary queue}, $\mathrm{Q_A}$. This queue does not alter the system dynamics but provides a clear framework for tracking the number and state of ongoing triggering delays.
We imagine that when a $\CP$ arrives at the \textit{main queue}, $\mathrm{Q_M}$, its corresponding $\CS$ simultaneously arrives at $\mathrm{Q_A}$. The ``service time'' in $\mathrm{Q_A}$ is exactly the random triggering delay $Y$. Upon service completion at $\mathrm{Q_A}$ (i.e., after time duration $Y$), the $\CS$ arrives at $\mathrm{Q_M}$. 

The memoryless property of the exponential distribution is the \emph{key} to our analysis, as we will see, it allows us to define a \emph{tractable} continuous-time Markov process based on the system workload and the number of $\CS$ customers in $\mathrm{Q_A}$, whose corresponding $\CP$ customers have arrived but whose triggering delays have not yet elapsed.

\subsection{State Space and Kolmogorov Equations}
\label{subsec:exp:state}

The main challenge in analyzing the $\mathrm{M^T/G/1}$ queue through the formulation of the underlying continuous-time Markov process is that the structure of the transition matrix becomes complicated due to the specific non-renewal of its aggregate arrival process. 
To address this issue, when the triggering delay follows an exponential distribution, we define an appropriate state descriptor that captures all the information required to characterize the system at time~$t$ and to predict its evolution over the infinitesimal interval~$(t, t+\Delta t]$ in a compact and analytically tractable form.
We describe the system state at time~$t$ by the pair $\big(N_{\mathrm{A}}(t), V_{\mathrm{M}}(t)\big)$, where:
\begin{itemize}
    \item $N_{\mathrm{A}}(t) \in \{0, 1, 2, \dots\}$ denotes the number of $\CS$ customers currently undergoing their triggering delay in $\mathrm{Q_A}$.
    \item $V_{\mathrm{M}}(t) \in [0, \infty)$ denotes the \emph{workload} of $\mathrm{Q_M}$, representing the total remaining service time of all customers in $\mathrm{Q_M}$ at time~$t$.
\end{itemize}

With this state representation, the stochastic process $\{(N_{\mathrm{A}}(t), V_{\mathrm{M}}(t)); t \ge 0\}$ constitutes a continuous-time Markov process, because the current state $(N_{\mathrm{A}}(t), V_{\mathrm{M}}(t))$ contains all information required to probabilistically describe the evolution of the system in the time window $(t, t + \Delta t]$.\footnote{As a sketch of the proof, consider an infinitesimal interval $(t, t + \Delta t]$ as $\Delta t\downarrow 0$, during which one of the following events may occur: (i) a $\CP$ arrival occurs with probability $\lambda\Delta t + o(\Delta t)$; (ii) a $\CS$ arrival (i.e., a service completion in $\mathrm{Q_A}$) occurs with probability $\gamma N_{\mathrm{A}}(t)\Delta t + o(\Delta t)$; and (iii) no arrival occurs with the complementary probability. 
Each of these transition probabilities depends only on the current state through $N_{\mathrm{A}}(t)$ and on the parameters $\lambda$ and $\gamma$. Furthermore, the service times of $\CP$ and $\CS$ customers (which determine the workload of $\mathrm{Q_M}$) are \emph{independent} of the system history.
Therefore, the future evolution of the process depends solely on its current state, thereby satisfying the Markov property.} 

It is crucial to note that if the triggering delays follow a general distribution, this state representation is no longer sufficient. 
In this case, one would need to track the residual triggering delay for each $\CS$ in $\mathrm{Q_A}$. 
The assumption of exponential delays reduces this complexity to a single integer descriptor, $N_{\mathrm{A}}(t)$.

Our goal is to characterize the steady-state
behavior of the system.
Let the joint time-dependent distribution of the process be defined, for $ x \ge 0, n \ge 0, t \ge 0$, as
\begin{equation*}
    F_t(n,x) := \Prob \big(N_{\mathrm{A}}(t) = n,  V_{\mathrm{M}}(t) \le x\big).
\end{equation*}

The distribution function $F_t(n,x)$ contains a probability mass (atom) at $x=0$, corresponding to the idle main queue $\mathrm{Q_M}$, which must be addressed separately.
Accordingly, we define
\begin{equation*}
    P_t(n) \coloneqq \Prob  \big( N_{\mathrm{A}}(t)=n, V_{\mathrm{M}}(t)=0 \big),
\end{equation*}
and for the continuous part of the distribution, where the workload is positive ($x>0$), we introduce the probability density function
\begin{equation*}
    f_t(n,x) \coloneqq \frac{\partial}{\partial x} F_t(n,x).
\end{equation*}
Hence for $x>0$,
\begin{equation}\label{eq:exp:pdf_def}
    F_t(n,x) = P_t(n) + \int_{0^+}^{x} f_t(n,u)  \diff u.
\end{equation}

To derive the forward Chapman-Kolmogorov equations governing this distribution, we examine the possible state transitions over an infinitesimal interval $(t, t+\Delta t]$ in the regime $\Delta t\downarrow 0$.
As mentioned, for $x>0$, the state $\big(N_{\mathrm{A}}(t+\Delta t) = n, V_{\mathrm{M}}(t+\Delta t) \le x\big)$ can be reached through the following events:
\begin{itemize}
    \item \emph{No arrival}: With probability $1-(\lambda+n\gamma)\Delta t+o(\Delta t)$, neither a $\CP$ nor a $\CS$ arrives at $\mathrm{Q_M}$. In this case, the workload decreases linearly by $\Delta t$. For the workload to be at most $x$ at time $t+\Delta t$, it must have been at most $x+\Delta t$ at time $t$. The contribution from this event is
    \begin{equation*}
        \big(1-(\lambda+n\gamma)\Delta t\big)F_t(n,x+\Delta t) + o(\Delta t).
    \end{equation*}
    
    \item \emph{Primary customer arrival}: A $\CP$ arrives with probability $\lambda\Delta t+o(\Delta t)$. 
    This event increases the number of $\CS$ customers in $\mathrm{Q_A}$ from $n-1$ to $n$.
    The workload also increases by the service time of the new $\CP$ arrival (i.e., $X_{\CP}$).
    The contribution is
    \begin{equation*}
        \lambda\Delta t \mathbbm{1}_{\{n\ge1\}}\int_{0}^{x} F_t(n-1,x-y)  \diff\BP(y) + o(\Delta t).
    \end{equation*}
    
    \item \emph{Secondary customer arrival}: If the system had $n+1$ pending $\CS$ customers at time $t$, one of them can complete its delay and arrive at the queue with probability $(n+1)\gamma\Delta t+o(\Delta t)$. 
    This event decreases the number of pending $\CS$ customers from $n+1$ to $n$, and the workload increases by the service time of the arriving $\CS$ (i.e., $X_{\CS}$).
    The contribution is
    \begin{equation*}
        (n+1)\gamma\Delta t \int_{0}^{x} F_t(n+1,x-y)  \diff\BS(y) + o(\Delta t).
    \end{equation*}
\end{itemize}

By summing these mutually exclusive contributions, we obtain the balance equation for the state distribution as follows:
\begin{equation*}
    \begin{split}
          F_{t+\Delta t}(n,x)
          &= \big(1-(\lambda+n\gamma)\Delta t\big)F_t(n, x+\Delta t)\\
          &\quad + \lambda\Delta t\ \mathbbm{1}_{\{n\ge1\}}\int_{0}^{x} F_t(n-1,x-y)  \diff\BP(y)\\
          &\quad + (n+1)\gamma\Delta t\int_{0}^{x} F_t(n+1,x-y)  \diff\BS(y)\\  
          &\quad+ o(\Delta t),\qquad (x>0).
    \end{split}
\end{equation*}
By rearranging, dividing by $\Delta t$, and taking the limit as $\Delta t \downarrow 0$, we obtain the following system of partial integro-differential equations: for $t, x>0$ and $n=0,1,2,\ldots$, 
\begin{equation*}
    \begin{split}
          \frac{\partial}{\partial t}F_t(n,x) &- \frac{\partial}{\partial x}F_t(n,x) 
          = -(\lambda+n\gamma)F_t(n,x) \\ &\quad + \lambda\mathbbm{1}_{\{n\ge1\}}\int_{0}^{x} F_t(n-1,x-y)  \diff\BP(y)\\
          &\quad + (n+1)\gamma\int_{0}^{x} F_t(n+1,x-y)  \diff\BS(y).
    \end{split}
\end{equation*}
It is observed that the last two terms on the right-hand side involve convolution operators.

We are particularly interested in the steady-state regime.
Assuming the system is stable and a steady-state distribution $F(n,x) \coloneqq \lim_{t\to\infty}F_t(n,x)$ exists, the time derivative vanishes.
This yields the final system of integro-differential equations that governs the steady-state behavior for $x>0$ and $n=0,1,2,\ldots$:
\begin{equation}\label{eq:exp:ide_steady}
   \begin{split}
        -\frac{\partial}{\partial x}F(n,x) &= -(\lambda+n\gamma)F(n,x)\\  &\quad+ \lambda\mathbbm{1}_{\{n\ge1\}}\int_{0}^{x} F(n-1,x-y)  \diff\BP(y)\\
        &\quad + (n+1)\gamma\int_{0}^{x} F(n+1,x-y)  \diff\BS(y).
   \end{split} 
\end{equation}
This equation completely describes the steady-state behavior of the workload of $\mathrm{Q_M}$ and the number of $\CS$ customers in $\mathrm{Q_A}$.
With the governing integro-differential equations established, the next step is to transform them into a tractable algebraic system.

\subsection{Transform Domain Analysis}
\label{subsec:exp:transform}

To solve the steady-state equations derived in \eqref{eq:exp:ide_steady}, we determine the LST with respect to the workload variable $x$, so as to convert our system of integro-differential equations and convolution operators into a system of algebraic equations.

Let $P_n \coloneqq \Prob  (N_{\mathrm{A}}=n, V_{\mathrm{M}}=0)$ be the steady-state probability that $\mathrm{Q_M}$ is idle when there are $n$ customers in $\mathrm{Q_A}$. 
For each $n \ge 0$, we define
\begin{equation*}
    \widetilde{\phi} _n(s) \coloneqq \E\big[e^{-sV_{\mathrm{M}}}\mathbbm{1}_{\{N_{\mathrm{A}}=n\}}\big]
    = P_n + \int_{0^+}^{\infty} e^{-s x} f(n,x)  \diff x,
\end{equation*}
for $\RealPart(s)\ge0$, where $f(n,x)$ is the probability density function for $x>0$.
To proceed, we leverage the following standard LST properties for \eqref{eq:exp:pdf_def} with the transform $\widetilde{\phi} _n(s)$:
\begin{itemize}
    \item The LST of $f(n,x)$ becomes
    \begin{equation*}
        \int_{0^+}^\infty e^{-sx}f(n,x)  \diff x = \widetilde{\phi} _n(s)-P_n.
    \end{equation*}
    \item The LST of $F(n,x)$ is related to $\widetilde{\phi} _n(s)$ by
    \begin{equation*}
        \int_{0}^\infty e^{-sx} F(n,x)  \diff x = \frac{\widetilde{\phi} _n(s)}{s}.
    \end{equation*}
    \item  The LST of a convolution yields
    \begin{equation*}
        \int_{0}^\infty  e^{-sx} \!\left(\int_{0}^{x} F(k,x-y)  \diff B_j(y)\!\right) \diff x
    =\! \frac{\widetilde{\phi} _{k}(s)}{s}\! \widetilde{B} _j(s),
    \end{equation*}
    where $j \in \{\CP, \CS\}$ and $\widetilde{B} _j(s)$ is the LST of the corresponding service time distribution of $\CP$ and $\CS$ customers.
\end{itemize}
Applying these properties to each term in the steady-state equation~\eqref{eq:exp:ide_steady}, for each $n \ge 0$,
\begin{equation*}
    \begin{split}
    - \big(\widetilde{\phi} _n(s)-P_n\big) &= -(\lambda+n\gamma)\frac{\widetilde{\phi} _n(s)}{s}\\ &\quad + \lambda\mathbbm{1}_{\{n\ge1\}}\frac{\widetilde{\phi} _{n-1}(s)}{s}\LBP(s) \\ &\quad + (n+1)\gamma\frac{\widetilde{\phi} _{n+1}(s)}{s}\LBS(s).
    \end{split}
\end{equation*}
Multiplying the entire equation by $s$ and rearranging the terms to group the unknown functions $\{\widetilde{\phi} _n(s)\}$, we arrive at an infinite-dimensional system of linear algebraic equations that relates the LSTs and the unknown idle probabilities $\{P_n\}$. This system consists of the following equations: for $n=0$,
\begin{equation}\label{eq:exp:sys_n0}
      (s-\lambda)\widetilde{\phi} _0(s) + \gamma\LBS(s)\widetilde{\phi} _1(s) = sP_0,
\end{equation}
for $n>0$,
\begin{equation}\label{eq:exp:sys_n}
    \begin{split}
        \lambda\LBP(s)\widetilde{\phi} _{n-1}(s) &+ (s-\lambda-n\gamma)\widetilde{\phi} _n(s)\\&+ (n+1)\gamma\LBS(s)\widetilde{\phi} _{n+1}(s) = sP_n.
    \end{split}
\end{equation}
We have successfully transformed the problem, but now face a new challenge: solving this infinite system of algebraic equations.
The following subsection addresses this by truncation with a proper boundary condition.

\subsection{Finite-State Matrix Formulation}
\label{subsec:exp:matrix}

The infinite system in~\eqref{eq:exp:sys_n0}--\eqref{eq:exp:sys_n} cannot be solved directly, as it forms an unbounded recurrence with state-dependent coefficients. To make it tractable, we truncate the state space of $\mathrm{Q_A}$ at a sufficiently large $\nmax$, such that the probability that the number of $\CS$ customers in $\mathrm{Q_A}$ exceeds $\nmax$ is negligible. This reduces the problem to a finite system of linear equations solvable via standard matrix-analytic methods.

The truncation requires a specific boundary condition at $n=\nmax$. 
We assume that when $N_{\mathrm{A}}=\nmax$, any new $\CS$ is blocked from entering $\mathrm{Q_A}$.
However, its corresponding $\CP$ is \emph{not} blocked and still joins $\mathrm{Q_M}$.
This means that a $\CP$ arrival at state $\nmax$ increases the workload but leaves the state of $\mathrm{Q_A}$ unchanged.
This boundary condition modifies the steady-state integro-differential equation~\eqref{eq:exp:ide_steady} for the specific case of $n=\nmax$ as follows:
\begin{equation*}
    \begin{split}
          -\frac{\partial}{\partial x}F(\nmax,x) &= -(\lambda+\nmax\gamma)F(\nmax,x) \\ &\quad  + \lambda\int_{0}^{x} F(\nmax-1,x-y)  \diff\BP(y)\\
          &\quad + \lambda\int_{0}^{x} F(\nmax,x-y)  \diff\BP(y).
    \end{split}
\end{equation*}
The final term reflects the self-transition at state $\nmax$ due to a $\CP$ arrival.
Applying the LST to this boundary equation yields the final row of our linear system for $n=\nmax$,
\begin{equation}\label{eq:exp:sys_nmax}
    \begin{split}
         \big(s - \nmax\gamma -& \lambda(1 - \LBP(s))\big)\widetilde{\phi} _{\nmax}(s)\\ &\qquad+\lambda\LBP(s)\widetilde{\phi} _{\nmax-1}(s) = sP_{\nmax}.
    \end{split}
\end{equation}
With this finite state space $\{0, 1, \dots, \nmax\}$, the resulting system of $\nmax+1$ linear equations can be expressed compactly in matrix form as
\begin{equation}
    {\boldsymbol M}(s) {\boldsymbol \Phi}(s) = s {\boldsymbol p},
    \label{eq:exp:matrix_sys}
\end{equation}
where ${\boldsymbol \Phi}(s) = [\widetilde{\phi} _0(s), \widetilde{\phi} _1(s), \dots, \widetilde{\phi} _{\nmax}(s)]^{\transpose}$ is the vector of unknown transforms, and ${\boldsymbol p} = [P_0, P_1, \dots, P_{\nmax}]^{\transpose}$ is the vector of idle probabilities.

The $(\nmax+1)\times(\nmax+1)$ characteristic matrix
${\boldsymbol M}(s)$ captures the system dynamics.
Its structure is tridiagonal, with a modified final row reflecting
the boundary condition at state $\nmax$:
\begin{equation*}
\resizebox{\columnwidth}{!}{$
{\boldsymbol M}(s)\! \!:=\! \!
\begin{pmatrix}
s-\lambda & \gamma\LBS(s) & \cdots & 0 \\[3pt]
\lambda\LBP(s) & s-\lambda-\gamma & \cdots & 0 \\[3pt]
0 & \lambda\LBP(s) & \ddots & \vdots \\[3pt]
\vdots & \ddots & \ddots & \nmax\gamma\LBS(s) \\[3pt]
0 & \cdots & \lambda\LBP(s) &
s\!-\!\nmax\gamma\!-\!\lambda\!\left(\!1\!-\!\LBP(s)\!\right)
\end{pmatrix}
$}
\label{eq:M_matrix}
\end{equation*}
The formal solution for the LST vector ${\boldsymbol \Phi}(s)$ is obtained by inverting the matrix ${\boldsymbol M}(s)$:
\begin{equation}\label{eq:exp:phi_sol}
    {\boldsymbol \Phi}(s) =  s  {\boldsymbol M}(s)^{-1} {\boldsymbol p} = s  \frac{\adj\big({\boldsymbol M}(s)\big) {\boldsymbol p}}{\Det\big({\boldsymbol M}(s)\big)};
\end{equation}
recall that the adjugate (or classical adjoint) $\adj({\boldsymbol A})$ of a square matrix ${\boldsymbol A}$ is the transpose of its cofactor matrix,  satisfying the property ${\boldsymbol A}^{-1} = ({\Det({\boldsymbol A})})^{-1}\adj({\boldsymbol A})$.
Conveniently, the tridiagonal structure of ${\boldsymbol M}(s)$ enables efficient and stable computation of both its determinant and adjugate via well-established recursive algorithms~\cite{press2007numerical}. Consequently, ${\boldsymbol \Phi}(s)$ can be determined once the unknown idle probability vector ${\boldsymbol p}$ is found, which we obtain next by enforcing the appropriate boundary conditions.

\begin{remark}
The choice of truncation level $\nmax$ represents a fundamental trade-off between model accuracy and computational cost.
A practical and robust guideline is to select $\nmax$ such that the probability of exceeding this level in the corresponding \emph{unbounded} system is negligible.
The number of customers in the untruncated auxiliary queue (which is the $\mathrm{M/M/}\infty$ queue) follows a Poisson distribution with mean $\bar\rho = \lambda/\gamma$. 
Therefore, $\nmax$ should be chosen to satisfy $\sum_{k=\nmax+1}^{\infty} {e^{-\bar\rho}\bar\rho^{k}}/{k!} < \epsilon$ for a small tolerance $\epsilon$, e.g., $10^{-6}$.
\end{remark}

\subsection{Boundary Conditions and Idle Probabilities}
\label{subsec:exp:boundary}

The solution in~\eqref{eq:exp:phi_sol}, which depends on the unknown idle probability vector ${\boldsymbol p} = [P_0, \dots, P_{\nmax}]^{\transpose}$, requires us to determine these $\nmax+1$ values. This is done by exploiting fundamental properties of LSTs and {truncation of} matrix-analytic methods. Specifically, for ${\boldsymbol \Phi}(s)$ to represent valid probability distributions, each component $\widetilde{\phi}_n(s)$ must be analytic in the open right half-plane, i.e., for $\RealPart(s) > 0$. From~\eqref{eq:exp:phi_sol}, the roots of the characteristic equation, $\Det({\boldsymbol M}(s)) = 0$, correspond to potential poles of ${\boldsymbol \Phi}(s)$. Analyticity requires that these poles be canceled by corresponding zeros in the numerator~\cite{40f1d0b7be7a48c9a05a521521c7d38f}.

The system under study is a classic example of a Markov Additive Process (MAP) with one-sided jumps, where the background Markov chain governs the number of $\CS$ customers in $\mathrm{Q_A}$.
The matrix ${\boldsymbol M}(s)$ is the Laplace exponent of this MAP.
The powerful results of~\cite{IVANOVS20101776} provide a rigorous foundation for this analysis, proving that under the stability condition $\rho < 1$, the characteristic equation $\Det({\boldsymbol M}(s)) = 0$ has precisely $\nmax + 1$ roots in the closed right half-plane, $\RealPart(s) \geq 0$.

One of these roots is always located at the origin, i.e., $s_0 = 0$. 
For the remaining $\nmax$ roots, denoted as $s_1, s_2, \dots, s_{\nmax}$, the analyticity requirement imposes the following $\nmax$ conditions:
\begin{equation}
    \label{eq:exp:analyticity}
    s_k \adj\big({\boldsymbol M}(s_k)\big) {\boldsymbol p} = \mathbf{0}, \quad \text{for } k = 1, 2, \dots, \nmax.
\end{equation}
For any non-zero root $s_k$, the matrix ${\boldsymbol M}(s_k)$ is singular by definition.
If the roots are distinct, the rank of ${\boldsymbol M}(s_k)$ is $\nmax$, which implies that the rank of its adjugate is 1.
Consequently, all rows of $\adj({\boldsymbol M}(s_k))$ are linearly dependent, and each root $s_k$ provides exactly one independent linear constraint on the elements of ${\boldsymbol p}$. 
Thus, the $\nmax$ roots yield $\nmax$ homogeneous equations. Since ${\boldsymbol p}$ has $n_{\max}+1$ components, one additional independent condition is required to determine a unique solution; this condition is obtained from the boundary behavior at $s=0$ (see~\cite{neuts1981matrix}).

The final independent equation is derived from the system behavior at $s=0$.
Differentiating the main matrix equation~\eqref{eq:exp:matrix_sys} with respect to $s$ gives
\begin{equation*}
    {\boldsymbol M}'(s) {\boldsymbol \Phi}(s) + {\boldsymbol M}(s) {\boldsymbol \Phi}'(s) = {\boldsymbol p}.
\end{equation*}
Evaluating at $s=0$ yields a linear system for the unknown derivative vector ${\boldsymbol \Phi}'(0)$:
\begin{equation}
    \label{eq:exp:diff_zero}
    {\boldsymbol M}(0) {\boldsymbol \Phi}'(0) = {\boldsymbol p} - {\boldsymbol M}'(0) {\boldsymbol \Phi}(0).
\end{equation}
Crucially, the vector ${\boldsymbol \Phi}(0) = [\widetilde{\phi} _0(0), \dots, \widetilde{\phi} _{\nmax}(0)]^{\transpose}$ is known.
Since $\widetilde{\phi} _n(0) = \E[e^{-0 \cdot V} \mathbbm{1}_{\{N_{\mathrm{A}} = n\}}] = \Prob  (N_{\mathrm{A}} = n)$, which is the marginal steady-state probability distribution of the number of $\CS$ customers in $\mathrm{Q_A}$.
This process $\{N_{\mathrm{A}}(t)\}$ evolves as a birth-death process on $\{0, \dots, \nmax\}$ with birth rate $\lambda$ and state-dependent death rate $n\gamma$.
This is an $\mathrm{M/M/}\nmax/\nmax$ loss system with load $\bar\rho = \lambda / \gamma$, whose stationary distribution is 
\begin{equation}
    \label{eq:exp:qa_dist}
    \widetilde{\phi} _n(0) = \Prob  (N_{\mathrm{A}} = n) = \frac{\bar\rho^n / n!}{\sum_{k=0}^{\nmax} \bar\rho^k / k!},
\end{equation}
for $ n = 0, 1, \dots, \nmax$.
The matrix ${\boldsymbol M}(0)$ is singular, as it is the transpose of the infinitesimal generator matrix for the underlying finite-state birth-death process governing $N_{\mathrm{A}}(t)$.
For such a generator matrix, the sum of each row is zero, implying that the column vector of all ones is a right null vector.
Consequently, for ${\boldsymbol M}(0)$, the row vector of all ones, $\mathbf{1}^{\transpose} = [1, 1, \dots, 1]$, is its left null-vector corresponding to the eigenvalue $0$, i.e., $\mathbf{1}^{\transpose}{\boldsymbol M}(0) = \mathbf{0}^{\transpose}$.
For the linear system~\eqref{eq:exp:diff_zero} to have a solution for ${\boldsymbol \Phi}'(0)$, the right-hand side must be orthogonal to this left null vector.
This solvability condition yields the following final equation:
\begin{equation}
    \label{eq:exp:norm_cond}
    \mathbf{1}^{\transpose} \big({\boldsymbol p} - {\boldsymbol M}'(0) {\boldsymbol \Phi}(0)\big) = 0 \implies \mathbf{1}^{\transpose} {\boldsymbol p} = \mathbf{1}^{\transpose} {\boldsymbol M}'(0) {\boldsymbol \Phi}(0).
\end{equation}
The vector ${\boldsymbol \Phi}(0)$ is the known steady-state probability distribution of the number of customers in $\mathrm{Q_A}$, which evolves as an $\mathrm{M/M/}\nmax/\nmax$ loss system.
Combining the $\nmax$ equations from~\eqref{eq:exp:analyticity} with~\eqref{eq:exp:norm_cond} forms a full-rank system that uniquely determines ${\boldsymbol p}$. 
Having fully characterized the system state LSTs, we now proceed to extract the performance metrics, beginning with the workload moments.

\subsection{Workload Moments}
\label{subsec:exp:moments}

With the idle probability vector ${\boldsymbol p}$ determined, the LST vector ${\boldsymbol \Phi}(s)$ is fully specified.
We can now compute the key performance metrics, starting with the moments of the steady-state workload, $V$.
The moments are obtained from the derivatives of the total workload LST, $\E[e^{-sV}] = \mathbf{1}^{\transpose}{\boldsymbol \Phi}(s)$, evaluated at $s=0$:
\begin{equation*}
    \E[V_{\mathrm{M}}^k] = (-1)^k \frac{\diff ^k}{\diff s^k} \E[e^{-sV_{\mathrm{M}}}] \bigg|_{s=0} = (-1)^k \mathbf{1}^{\transpose} {\boldsymbol \Phi}^{(k)}(0).
\end{equation*}
Our task, therefore, is to compute these derivative vectors.

\subsubsection{Mean Workload}
The mean workload is given by $\E[V_{\mathrm{M}}] = -\mathbf{1}^{\transpose} {\boldsymbol \Phi}'(0)$. The vector of first derivatives, ${\boldsymbol \Phi}'(0)$, is the solution to the linear system~\eqref{eq:exp:diff_zero} derived in the previous section. However, the primary challenge in solving~\eqref{eq:exp:diff_zero} is that the coefficient matrix ${\boldsymbol M}(0)$ is singular. This implies that the system does not have a unique solution without an additional constraint. The required constraint is obtained by differentiating the main matrix equation~\eqref{eq:exp:matrix_sys} twice with respect to $s$:
\begin{equation*}
    {\boldsymbol M}''(s) {\boldsymbol \Phi}(s) + 2 {\boldsymbol M}'(s) {\boldsymbol \Phi}'(s) + {\boldsymbol M}(s) {\boldsymbol \Phi}''(s) = \mathbf{0}.
\end{equation*}
The key insight is to evaluate this equation at $s=0$ and left-multiply by the left null vector of ${\boldsymbol M}(0)$, which is $\mathbf{1}^{\transpose}$. Since $\mathbf{1}^{\transpose}{\boldsymbol M}(0) = \mathbf{0}^{\transpose}$, this step elegantly eliminates the unknown second derivative term ${\boldsymbol \Phi}''(0)$:
\begin{equation} \label{eq:exp:moments_constraint}
    \mathbf{1}^{\transpose} \big( {\boldsymbol M}''(0) {\boldsymbol \Phi}(0) + 2{\boldsymbol M}'(0) {\boldsymbol \Phi}'(0) \big) = 0.
\end{equation}
By replacing any one of the linearly dependent equations in the singular system~\eqref{eq:exp:diff_zero} with this new constraint~\eqref{eq:exp:moments_constraint}, we obtain an augmented system with full rank that can be uniquely solved for ${\boldsymbol \Phi}'(0)$. The mean workload is then readily computed as $\E[V_{\mathrm{M}}] = -\mathbf{1}^{\transpose}{\boldsymbol \Phi}'(0)$.

\subsubsection{Higher Moments}
This differentiation-based approach can be systematically extended to determine any higher moment. For instance, the second moment is $\E[V^2] = \mathbf{1}^{\transpose}{\boldsymbol \Phi}''(0)$. To find the vector ${\boldsymbol \Phi}''(0)$, we first rearrange the second-derivative equation evaluated at $s=0$:
\begin{equation*}
    {\boldsymbol M}(0){\boldsymbol \Phi}''(0) = - {\boldsymbol M}''(0){\boldsymbol \Phi}(0) - 2{\boldsymbol M}'(0){\boldsymbol \Phi}'(0). 
\end{equation*}
The right-hand side is now fully determined, as ${\boldsymbol \Phi}'(0)$ was computed in the previous step. This remains a singular system, so to obtain the additional constraint, we differentiate the main matrix equation a third time, evaluate at $s=0$, and left-multiply by $\mathbf{1}^{\transpose}$ to eliminate ${\boldsymbol \Phi}'''(0)$. The resulting full-rank system can then be uniquely solved for ${\boldsymbol \Phi}''(0)$.

This framework can be applied iteratively to compute moments of any order, with the $k$-th moment depending on all lower moments from $0$ to $k-1$. While workload moments characterize the system state, they do not directly capture individual customer waiting times --- a distinction we address in the next subsection.

\subsection{Class-Specific Waiting and System Times}
\label{subsec:exp:metrics}

With the system workload fully characterized by the LST vector ${\boldsymbol \Phi}(s)$ and the idle probability vector ${\boldsymbol p}$, we can now derive key performance metrics for each customer class, focusing on the distributions of waiting time, $W$, and system time, $S$. The analytical approach differs between classes due to the nature of their arrivals. A central element is the PASTA property~\cite{Wolff1982Poisson}, which asserts that Poisson arrivals observe the system in steady state. This relies on the \emph{lack of anticipation} property: the occurrence of an arrival at time $t$ reveals no information about the system state immediately before the arrival. To apply the PASTA property correctly, we must refine the notion of independence underlying it, particularly when multiple interacting arrival streams are present. Following the definition from our work~\cite{11571174}.

\begin{definition}[\emph{Causally independent} Poisson arrival process] \label{def:exp:causal_indep}
  An arrival process for a class of customers is a \emph{causally independent Poisson process} if it is a Poisson process whose arrival epochs are statistically independent of the history and current state of the queueing system, including the arrival times of the other customer classes that interact with the system.
\end{definition}

This definition emphasizes independence not only from system-state evolution (such as the number of customers in the system or server status) but also from the timing of other events (such as arrivals of other classes) that directly affect system evolution up to the arrival instant. Using this definition, we assess the applicability of PASTA in our M$^{\text{T}}$/G/1 model:

\emph{Primary Customers}: The arrival process of $\CP$ customers is, by definition, a homogeneous Poisson process. These arrivals are generated externally and independently of the state of the queue, the service process, and the arrival times of the $\CS$ customers. This perfectly matches the criteria of a \emph{causally independent} Poisson process. Therefore, the PASTA property holds: Arriving $\CP$ customers see the queue in its steady-state distribution.

\emph{Secondary Customers}: From Lemma~\ref{lem:model:arrival_props}, we know that although the marginal arrival process of $\CS$ customers is Poisson, it is \emph{not} causally independent. Each $\CS$ arrival epoch is explicitly determined by a past $\CP$ arrival time plus a random delay $Y$. This dependency violates the lack of anticipation condition. The arrival of a $\CP$ necessarily alters the system state $Y$ time units \emph{before} the corresponding $\CS$ arrives. Thus, knowing a $\CS$ arrival implies knowledge about a related past event (a $\CP$ arrival) that influenced the trajectory of the queue up to the $\CS$ arrival instant. Therefore, PASTA does \emph{not} directly apply: Arriving $\CS$ customers do not necessarily see the queue in its steady-state distribution.

\subsubsection{Performance of Primary Customers}

As established by PASTA, an arriving $\CP$ customer sees the system in its steady-state distribution. Therefore, the waiting time that a primary customer experiences, $W_{\CP}$, is equal in distribution to the steady-state workload, $V_{\mathrm{M}}$. The LST of the waiting time for a $\CP$ customer, denoted by $\widetilde{W} _{\CP}(s)$, is the LST of the total workload, obtained by summing the components of our solution vector ${\boldsymbol \Phi}(s)$:
\begin{equation*}
   \widetilde{W} _{\CP}(s) \coloneqq \E[e^{-s W_{\CP}}] = \E[e^{-sV_{\mathrm{M}}}] = \sum_{n=0}^{\nmax} \widetilde{\phi} _n(s) = \mathbf{1}^{\transpose}{\boldsymbol \Phi}(s).
\end{equation*}
The moments of the waiting time are found by differentiating this LST and evaluating at $s=0$. Specifically, the mean and second moments are
\begin{align*}
    \E[W_{\CP}] &= \E[V_{\mathrm{M}}] = -\mathbf{1}^{\transpose}{\boldsymbol \Phi}'(0) ,\\
    \E[W_{\CP}^2] &= \E[V_{\mathrm{M}}^2] = \mathbf{1}^{\transpose}{\boldsymbol \Phi}''(0).
\end{align*}
The mean system time, $\E[S_{\CP}]$, is the sum of the mean waiting and service times: $\E[S_{\CP}] = \E[W_{\CP}] + \E[X_{\CP}]$.
Since the service time is independent of the waiting time, the variance of the system time is the sum of the variances:
$\var[S_{\CP}] = \var[W_{\CP}] + \var[X_{\CP}]$, where $\var[W_{\CP}] = \E[W_{\CP}^2] - (\E[W_{\CP}])^2$.

\subsubsection{Performance of Secondary Customers}

Since PASTA does not apply, we must determine the state distribution seen by an arriving $\CS$ using an alternative method. The appropriate tool is the Palm distribution~\cite{baccelli2013elements}, which formalizes the idea that the probability of an arrival seeing the system in a particular state is proportional to the rate of arrivals into that state.

A $\CS$ arrives at $\mathrm{Q_M}$ when its triggering delay in $\mathrm{Q_A}$ is completed. If there are $n$ customers in $\mathrm{Q_A}$, the total rate of departures from $\mathrm{Q_A}$ (and thus arrivals at $\mathrm{Q_M}$) is $n\gamma$. The arrival-stationary probability of seeing $n$ customers in $\mathrm{Q_A}$ is therefore
\begin{equation*}
    \begin{split}
            \Prob   \big(N_{\mathrm{A}}=n \mid& \text{arrival of } \CS \text{ at } \mathrm{Q_M}\big)  
            =\\ &\frac{ \text{Arrival rate of } \CS \text{ at } \mathrm{Q_M} \text{ when } N_{\mathrm{A}}=n}{\text{Total arrival rate of } \CS \text{ at } \mathrm{Q_M}}\\ 
            \qquad&= \frac{n\gamma \cdot \Prob  (N_{\mathrm{A}}=n)}{\sum_{k=1}^{\nmax} k\gamma \cdot \Prob  (N_{\mathrm{A}}=k)}\\&= \frac{n  \widetilde{\phi} _n(0)}{\E[N_{\mathrm{A}}]},
    \end{split}
\end{equation*}
for $n \ge 1$, where $\E[N_{\mathrm{A}}] = \sum_{k=1}^{\nmax} k\widetilde{\phi} _k(0)$.
The LST of the waiting time for a $\CS$ customer, $\widetilde{W} _{\CS}(s)$, is the expected value of $e^{-sV_{\mathrm{M}}}$ conditioned on this arrival-stationary distribution:
\begin{align*}
    \widetilde{W} _{\CS}(s) &\coloneqq \E[e^{-s W_{\CS}}]\\&= \sum_{n=1}^{\nmax} \E  \left[e^{-sV_{\mathrm{M}}}\! \mid \!N_{\mathrm{A}}=n\right] \Prob    \big(N_{\mathrm{A}}=n \!\mid \!\text{arrival of } \CS\big) \nonumber \\
    &= \sum_{n=1}^{\nmax} \frac{\E[e^{-sV_{\mathrm{M}}} \mathbbm{1}_{\{N_{\mathrm{A}}=n\}}]}{\Prob  (N_{\mathrm{A}}=n)} \cdot \frac{n  \Prob  (N_{\mathrm{A}}=n)}{\E[N_{\mathrm{A}}]} \\&= \frac{1}{\E[N_{\mathrm{A}}]} \sum_{n=1}^{\nmax} n  \widetilde{\phi} _n(s).
\end{align*}
The moments of the waiting time for $\CS$ customers are derived by differentiating this LST:
\begin{align*}
    \E[W_{\CS}] &= -\frac{\diff }{\diff s}\widetilde{W} _{\CS}(s)\bigg|_{s=0} = \frac{-\sum_{n=1}^{\nmax} n\widetilde{\phi}^{\prime}_n(0)}{\E[N_{\mathrm{A}}]}, \\
    \E[W_{\CS}^2] &= \frac{\diff ^2}{\diff s^2}\widetilde{W} _{\CS}(s)\bigg|_{s=0} = \frac{\sum_{n=1}^{\nmax} n\widetilde{\phi}^{\prime \prime}_n(0)}{\E[N_{\mathrm{A}}]}.
\end{align*}
These results are consistent with the previously derived formulas and are now explicitly expressed in terms of the LST of the waiting time distribution. The mean system time and variance for $\CS$ customers are given by $\E[S_{\CS}] = \E[W_{\CS}] + \E[X_{\CS}]$ and $\var[S_{\CS}] = \var[W_{\CS}] + \var[X_{\CS}]$, respectively. 
This completes the derivation of the main performance metrics for both customer classes under the assumption of exponential triggering delays.

\section{Numerical Results}
\label{sec:numerical}

\begin{table*}[t]
\centering
\caption{Absolute relative error (\%) between analytical and simulation results for the exponential triggering-delay assumption.}
\label{tab:comprehensive_errors}
\footnotesize
\setlength{\tabcolsep}{4.2pt}
\begin{tabular}{lcccccccccccccccccc}
\toprule
\rowcolor{gray!20}
\textbf{Metric} & \multicolumn{6}{c}{\textbf{$\rho=0.30$}} & \multicolumn{6}{c}{\textbf{$\rho=0.60$}} & \multicolumn{6}{c}{\textbf{$\rho=0.90$}} \\
\cmidrule(lr){2-7} \cmidrule(lr){8-13} \cmidrule(lr){14-19}
\rowcolor{gray!10}
$c^2_X$ & \multicolumn{2}{c}{$0.25$} & \multicolumn{2}{c}{$1.00$} & \multicolumn{2}{c}{$2.25$} & \multicolumn{2}{c}{$0.25$} & \multicolumn{2}{c}{$1.00$} & \multicolumn{2}{c}{$2.25$} & \multicolumn{2}{c}{$0.25$} & \multicolumn{2}{c}{$1.00$} & \multicolumn{2}{c}{$2.25$} \\
\cmidrule(lr){2-3} \cmidrule(lr){4-5} \cmidrule(lr){6-7} \cmidrule(lr){8-9} \cmidrule(lr){10-11} \cmidrule(lr){12-13} \cmidrule(lr){14-15} \cmidrule(lr){16-17} \cmidrule(lr){18-19}
\rowcolor{gray!10}
$\E[Y]$ & $1.00$ & $10.00$ & $1.00$ & $10.00$ & $1.00$ & $10.00$ & $1.00$ & $10.00$ & $1.00$ & $10.00$ & $1.00$ & $10.00$ & $1.00$ & $10.00$ & $1.00$ & $10.00$ & $1.00$ & $10.00$ \\
\midrule
$\E[S_{\CP}]$  & 0.03 & 0.03 & 0.05 & 0.00 & 0.02 & 0.01 & 0.07 & 0.03 & 0.09 & 0.02 & 0.02 & 0.07 & 0.19 & 0.67 & 0.31 & 0.07 & 0.24 & 0.02 \\
$\var[S_{\CP}]$ & 0.09 & 0.14 & 0.28 & 0.02 & 0.24 & 0.02 & 0.33 & 0.09 & 0.12 & 0.21 & 0.08 & 0.29 & 0.31 & 1.08 & 0.22 & 0.15 & 1.26 & 0.69 \\
$\E[S_{\CS}]$  & 0.03 & 0.03 & 0.03 & 0.00 & 0.02 & 0.01 & 0.07 & 0.03 & 0.08 & 0.02 & 0.02 & 0.07 & 0.19 & 0.64 & 0.31 & 0.07 & 0.23 & 0.02 \\
$\var[S_{\CS}]$ & 0.01 & 0.14 & 0.17 & 0.10 & 0.24 & 0.02 & 0.47 & 0.09 & 0.12 & 0.21 & 0.08 & 0.26 & 0.31 & 1.08 & 0.23 & 0.15 & 1.24 & 0.69 \\
\bottomrule
\end{tabular}
\vspace{1mm}

\parbox{0.98\textwidth}{\footnotesize
For a performance metric $M \in \{\E[S_{\CP}],\, \var[S_{\CP}],\, \E[S_{\CS}],\, \var[S_{\CS}]\}$, the relative error is defined as $|M^{\mathrm{Analytical}}-M^{\mathrm{Simulation}}|/ {M^{\mathrm{Simulation}}}\times 100\%$. The service-time distributions of $\CP$ and $\CS$ customers are assumed to be identical, with mean $\E[X]=1.00$ and squared coefficient of variation $c^2_X\in\{0.25,1.00,2.25\}$. The triggering delay has mean $\E[Y]\in\{1.00,10.00\}$ with squared coefficient of variation $c^2_Y=1.00$ (exponential distribution).}
\end{table*}

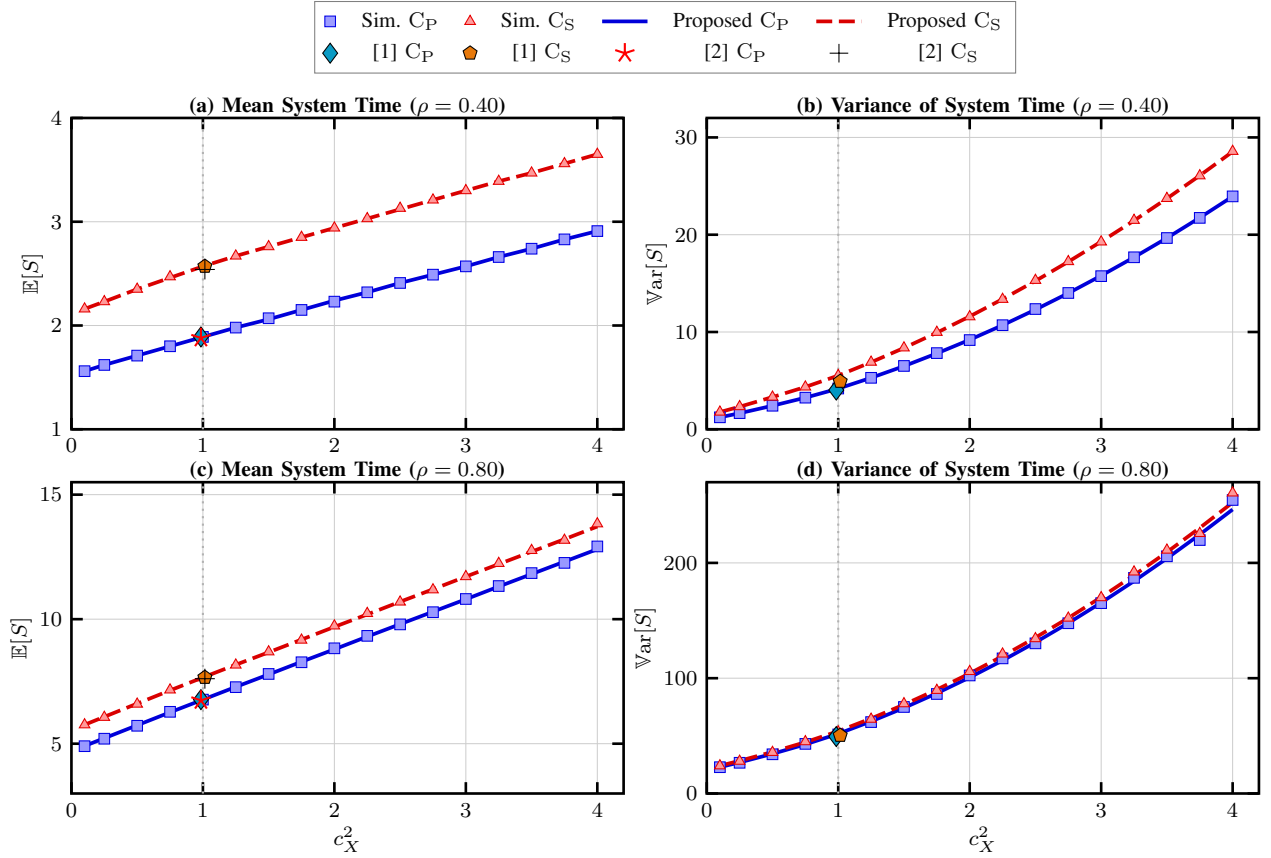
\begin{figure*}[t]
\centering

\begin{tikzpicture}
\begin{groupplot}[
    group style={
        group size=2 by 2,
        horizontal sep=1.1cm,
        vertical sep=0.7cm
    },
    width=0.49\textwidth,
    height=5.7cm,
    grid=both,
    grid style={line width=.2pt, draw=gray!25},
    major grid style={line width=.3pt, draw=gray!40},
    axis line style={line width=1pt},
    tick style={line width=1pt, color=black},
    xlabel near ticks,
    ylabel near ticks,
    xlabel shift=-3pt,
    ylabel shift=-4pt,
    title style={
        yshift=-9pt,
        font=\footnotesize\bfseries
    },
    tick label style={font=\footnotesize},
    label style={font=\footnotesize},
    legend style={
        font=\footnotesize,
        legend columns=4,
        fill=white,
        fill opacity=0.96,
        draw=gray,
        column sep=7pt,
        row sep=1pt
    }
]

% =====================================================================
% (a) Mean system time versus service-time SCV (rho = 0.40)
% =====================================================================
\nextgroupplot[
    ylabel={$\E[S]$},
    title={(a) Mean System Time ($\rho=0.40$)},
    xmin=0,
    xmax=4.2,
    ymin=1.0,
    ymax=4.0,
    legend to name=sharedlegend
]

% Exponential-service benchmark line
\draw[dotted, line width=0.8pt, draw=gray!60] (axis cs:1,1.0) -- (axis cs:1,4.0);

% Simulation: primary customers
\addplot[
    only marks,
    mark=square*,
    mark size=2.0pt,
    draw=blue!90!black,
    fill=blue!40
] coordinates {
    (0.10,1.56) (0.25,1.62) (0.50,1.71) (0.75,1.80)
    (1.00,1.89) (1.25,1.98) (1.50,2.07) (1.75,2.15)
    (2.00,2.23) (2.25,2.32) (2.50,2.41) (2.75,2.49)
    (3.00,2.57) (3.25,2.66) (3.50,2.74) (3.75,2.83)
    (4.00,2.91)
};
\addlegendentry{Sim. $\CP$}

% Simulation: secondary customers
\addplot[
    only marks,
    mark=triangle*,
    mark size=2.3pt,
    draw=red!90!black,
    fill=red!40
] coordinates {
    (0.10,2.16) (0.25,2.23) (0.50,2.35) (0.75,2.47)
    (1.00,2.57) (1.25,2.67) (1.50,2.76) (1.75,2.85)
    (2.00,2.94) (2.25,3.03) (2.50,3.13) (2.75,3.21)
    (3.00,3.30) (3.25,3.39) (3.50,3.47) (3.75,3.56)
    (4.00,3.65)
};
\addlegendentry{Sim. $\CS$}

% Proposed model: primary customers
\addplot[
    line width=1.4pt,
    solid,
    color=blue!85!black
] coordinates {
    (0.10,1.56) (0.25,1.62) (0.50,1.71) (0.75,1.80)
    (1.00,1.89) (1.25,1.98) (1.50,2.06) (1.75,2.15)
    (2.00,2.24) (2.25,2.32) (2.50,2.41) (2.75,2.49)
    (3.00,2.57) (3.25,2.66) (3.50,2.74) (3.75,2.83)
    (4.00,2.91)
};
\addlegendentry{Proposed $\CP$}

% Proposed model: secondary customers
\addplot[
    line width=1.4pt,
    dash pattern=on 6pt off 2.5pt,
    color=red!85!black
] coordinates {
    (0.10,2.16) (0.25,2.23) (0.50,2.35) (0.75,2.46)
    (1.00,2.57) (1.25,2.67) (1.50,2.76) (1.75,2.85)
    (2.00,2.94) (2.25,3.03) (2.50,3.12) (2.75,3.21)
    (3.00,3.30) (3.25,3.39) (3.50,3.47) (3.75,3.56)
    (4.00,3.65)
};
\addlegendentry{Proposed $\CS$}

% Prior-work legend entries
\addlegendimage{
    only marks,
    mark=diamond*,
    mark size=3.75pt,
    draw=black,
    fill=cyan!75!black
}
\addlegendentry{\cite{11571174} $\CP$}

\addlegendimage{
    only marks,
    mark=pentagon*,
    mark size=2.75pt,
    draw=black,
    fill=orange!90!black
}
\addlegendentry{\cite{11571174} $\CS$}

\addlegendimage{
    only marks,
    mark=star,
    mark size=3.8pt,
    line width=0.8pt,
    draw=red
}
\addlegendentry{
    \cite{Rahnamania2024CorrelatedArrivals} $\CP$
}

\addlegendimage{
    only marks,
    mark=+,
    mark size=3.5pt,
    draw=black,
    fill=purple!85!black
}
\addlegendentry{
    \cite{Rahnamania2024CorrelatedArrivals} $\CS$
}

% Prior [15] (INFOCOM)
\addplot[
    only marks,
    mark=diamond*,
    mark size=3.75pt,
    draw=black,
    fill=cyan!75!black
] coordinates {
    (0.985,1.89)
};

\addplot[
    only marks,
    mark=pentagon*,
    mark size=2.75pt,
    draw=black,
    fill=orange!90!black
] coordinates {
    (1.015,2.57)
};

% Prior [14] (IWCIT)
\addplot[
    only marks,
    mark=star,
    mark size=3.8pt,
    line width=0.8pt,
    draw=red
] coordinates {
    (0.985,1.88)
};

\addplot[
    only marks,
    mark=+,
    mark size=3.75pt,
    draw=black
] coordinates {
    (1.015,2.54)
};

% =====================================================================
% (b) Variance of system time versus service-time SCV (rho = 0.40)
% =====================================================================
\nextgroupplot[
    ylabel={$\var[S]$},
    title={(b) Variance of System Time ($\rho=0.40$)},
    xmin=0,
    xmax=4.2,
    ymin=0,
    ymax=32
]

\draw[dotted, line width=0.8pt, draw=gray!60] (axis cs:1,0) -- (axis cs:1,32);

% Simulation CP
\addplot[
    only marks,
    mark=square*,
    mark size=2.0pt,
    draw=blue!90!black,
    fill=blue!40
] coordinates {
    (0.10,1.23) (0.25,1.66) (0.50,2.42) (0.75,3.25)
    (1.00,4.19) (1.25,5.29) (1.50,6.51) (1.75,7.83)
    (2.00,9.17) (2.25,10.70) (2.50,12.36) (2.75,14.01)
    (3.00,15.74) (3.25,17.68) (3.50,19.66) (3.75,21.73)
    (4.00,23.94)
};

% Simulation CS
\addplot[
    only marks,
    mark=triangle*,
    mark size=2.3pt,
    draw=red!90!black,
    fill=red!40
] coordinates {
    (0.10,1.77) (0.25,2.33) (0.50,3.31) (0.75,4.35)
    (1.00,5.53) (1.25,6.90) (1.50,8.37) (1.75,9.97)
    (2.00,11.57) (2.25,13.36) (2.50,15.31) (2.75,17.24)
    (3.00,19.26) (3.25,21.48) (3.50,23.73) (3.75,26.06)
    (4.00,28.57)
};

% Proposed CP
\addplot[
    line width=1.4pt,
    solid,
    color=blue!85!black
] coordinates {
    (0.10,1.23) (0.25,1.65) (0.50,2.42) (0.75,3.25)
    (1.00,4.18) (1.25,5.29) (1.50,6.49) (1.75,7.80)
    (2.00,9.20) (2.25,10.70) (2.50,12.30) (2.75,13.99)
    (3.00,15.79) (3.25,17.68) (3.50,19.66) (3.75,21.75)
    (4.00,23.93)
};

% Proposed CS
\addplot[
    line width=1.4pt,
    dash pattern=on 6pt off 2.5pt,
    color=red!85!black
] coordinates {
    (0.10,1.77) (0.25,2.32) (0.50,3.31) (0.75,4.35)
    (1.00,5.52) (1.25,6.89) (1.50,8.36) (1.75,9.93)
    (2.00,11.60) (2.25,13.38) (2.50,15.25) (2.75,17.22)
    (3.00,19.29) (3.25,21.45) (3.50,23.72) (3.75,26.08)
    (4.00,28.54)
};

% Prior [15] Variance
\addplot[
    only marks,
    mark=diamond*,
    mark size=3.75pt,
    draw=black,
    fill=cyan!75!black
] coordinates {
    (0.985,4.03)
};

\addplot[
    only marks,
    mark=pentagon*,
    mark size=2.75pt,
    draw=black,
    fill=orange!90!black
] coordinates {
    (1.015,4.92)
};

% =====================================================================
% (c) Mean system time versus service-time SCV (rho = 0.80)
% =====================================================================
\nextgroupplot[
    xlabel={$c_X^2$},
    ylabel={$\E[S]$},
    title={(c) Mean System Time ($\rho=0.80$)},
    xmin=0,
    xmax=4.2,
    ymin=3.0,
    ymax=15.5
]

\draw[dotted, line width=0.8pt, draw=gray!60] (axis cs:1,3.0) -- (axis cs:1,15.5);

% Simulation CP
\addplot[
    only marks,
    mark=square*,
    mark size=2.0pt,
    draw=blue!90!black,
    fill=blue!40
] coordinates {
    (0.10,4.90) (0.25,5.20) (0.50,5.72) (0.75,6.28)
    (1.00,6.77) (1.25,7.27) (1.50,7.80) (1.75,8.27)
    (2.00,8.83) (2.25,9.33) (2.50,9.79) (2.75,10.28)
    (3.00,10.81) (3.25,11.33) (3.50,11.85) (3.75,12.26)
    (4.00,12.92)
};

% Simulation CS
\addplot[
    only marks,
    mark=triangle*,
    mark size=2.3pt,
    draw=red!90!black,
    fill=red!40
] coordinates {
    (0.10,5.76) (0.25,6.06) (0.50,6.59) (0.75,7.16)
    (1.00,7.65) (1.25,8.16) (1.50,8.69) (1.75,9.16)
    (2.00,9.73) (2.25,10.23) (2.50,10.69) (2.75,11.18)
    (3.00,11.71) (3.25,12.24) (3.50,12.76) (3.75,13.17)
    (4.00,13.83)
};

% Proposed CP
\addplot[
    line width=1.4pt,
    solid,
    color=blue!85!black
] coordinates {
    (0.10,4.91) (0.25,5.22) (0.50,5.74) (0.75,6.26)
    (1.00,6.77) (1.25,7.28) (1.50,7.78) (1.75,8.29)
    (2.00,8.79) (2.25,9.30) (2.50,9.80) (2.75,10.30)
    (3.00,10.80) (3.25,11.31) (3.50,11.81) (3.75,12.31)
    (4.00,12.81)
};

% Proposed CS
\addplot[
    line width=1.4pt,
    dash pattern=on 6pt off 2.5pt,
    color=red!85!black
] coordinates {
    (0.10,5.76) (0.25,6.08) (0.50,6.61) (0.75,7.14)
    (1.00,7.66) (1.25,8.17) (1.50,8.67) (1.75,9.18)
    (2.00,9.69) (2.25,10.19) (2.50,10.70) (2.75,11.20)
    (3.00,11.71) (3.25,12.21) (3.50,12.71) (3.75,13.21)
    (4.00,13.72)
};

% Prior [15] (INFOCOM)
\addplot[
    only marks,
    mark=diamond*,
    mark size=3.75pt,
    draw=black,
    fill=cyan!75!black
] coordinates {
    (0.985,6.77)
};

\addplot[
    only marks,
    mark=pentagon*,
    mark size=2.75pt,
    draw=black,
    fill=orange!90!black
] coordinates {
    (1.015,7.66)
};

% Prior [14] (IWCIT)
\addplot[
    only marks,
    mark=star,
    mark size=3.8pt,
    line width=0.8pt,
    draw=red
] coordinates {
    (0.985,6.74)
};

\addplot[
    only marks,
    mark=+,
    mark size=3.75pt,
    draw=black
] coordinates {
    (1.015,7.61)
};

% =====================================================================
% (d) Variance of system time versus service-time SCV (rho = 0.80)
% =====================================================================
\nextgroupplot[
    xlabel={$c_X^2$},
    ylabel={$\var[S]$},
    title={(d) Variance of System Time ($\rho=0.80$)},
    xmin=0,
    xmax=4.2,
    ymin=0,
    ymax=270
]

\draw[dotted, line width=0.8pt, draw=gray!60] (axis cs:1,0) -- (axis cs:1,270);

% Simulation CP
\addplot[
    only marks,
    mark=square*,
    mark size=2.0pt,
    draw=blue!90!black,
    fill=blue!40
] coordinates {
    (0.10,22.81) (0.25,26.64) (0.50,34.01) (0.75,43.03)
    (1.00,51.23) (1.25,61.94) (1.50,75.04) (1.75,86.31)
    (2.00,102.38) (2.25,117.19) (2.50,130.22) (2.75,147.67)
    (3.00,165.10) (3.25,187.04) (3.50,205.72) (3.75,219.83)
    (4.00,254.57)
};

% Simulation CS
\addplot[
    only marks,
    mark=triangle*,
    mark size=2.3pt,
    draw=red!90!black,
    fill=red!40
] coordinates {
    (0.10,24.03) (0.25,28.02) (0.50,35.66) (0.75,44.96)
    (1.00,53.45) (1.25,64.45) (1.50,77.89) (1.75,89.47)
    (2.00,105.85) (2.25,121.01) (2.50,134.33) (2.75,152.14)
    (3.00,169.87) (3.25,192.17) (3.50,211.11) (3.75,225.55)
    (4.00,260.70)
};

% Proposed CP
\addplot[
    line width=1.4pt,
    solid,
    color=blue!85!black
] coordinates {
    (0.10,22.89) (0.25,26.96) (0.50,34.41) (0.75,42.58)
    (1.00,51.83) (1.25,62.54) (1.50,74.24) (1.75,86.95)
    (2.00,100.66) (2.25,115.36) (2.50,131.07) (2.75,147.78)
    (3.00,165.49) (3.25,184.20) (3.50,203.91) (3.75,224.62)
    (4.00,246.33)
};

% Proposed CS
\addplot[
    line width=1.4pt,
    dash pattern=on 6pt off 2.5pt,
    color=red!85!black
] coordinates {
    (0.10,24.10) (0.25,28.34) (0.50,36.07) (0.75,44.50)
    (1.00,54.04) (1.25,65.06) (1.50,77.09) (1.75,90.11)
    (2.00,104.14) (2.25,119.17) (2.50,135.20) (2.75,152.23)
    (3.00,170.26) (3.25,189.29) (3.50,209.33) (3.75,230.36)
    (4.00,252.39)
};

% Prior [15] Variance
\addplot[
    only marks,
    mark=diamond*,
    mark size=3.75pt,
    draw=black,
    fill=cyan!75!black
] coordinates {
    (0.985,49.39)
};

\addplot[
    only marks,
    mark=pentagon*,
    mark size=2.75pt,
    draw=black,
    fill=orange!90!black
] coordinates {
    (1.015,50.37)
};

\end{groupplot}

% Shared legend centered above the 2x2 grid
\node[
    anchor=south,
    inner sep=0pt
] at ($(group c1r1.north)!0.5!(group c2r1.north)+(0,0.58cm)$) {
    \pgfplotslegendfromname{sharedlegend}
};

\end{tikzpicture}
\caption{
Effect of service-time variability on the mean system time $\E[S]$ and system-time variance $\var[S]$ under load $\rho=0.40$ (top row) and $\rho=0.80$ (bottom row). Experiments use $\E[X_{\CP}]=\E[X_{\CS}]=1.00$ and $\E[Y]=1.00$, with $c_X^2\in[0.10,4.00]$. The proposed matrix-analytic model is compared with discrete-event simulation for primary and secondary customers. Prior approximations \cite{11571174,Rahnamania2024CorrelatedArrivals} are restricted to exponential service times (at $c_X^2=1.00$).}
\label{fig:numerical_scaling_grid}
\end{figure*}

In this section, we validate the accuracy of the truncation developed in Section~\ref{sec:exp} under the assumption that triggering delays follow an exponential distribution, by comparing its predictions against discrete-event simulations. 

All simulations are implemented in Python using the \texttt{SimPy} discrete-event engine and \texttt{NumPy} for generating random variables. To ensure high statistical confidence, each simulation run processes $2 \times 10^7$ customers, with the first and last $10^6$ events discarded to eliminate initialization and termination effects. For the analytical model, the truncation level $\nmax$ of the auxiliary queue is dynamically adjusted such that the tail probability satisfies $\Prob(N_{\mathrm{A}} > n_{\max}) < 10^{-9}$.

We consider the following distributional families for the random variables $X_{\CP}$ and $X_{\CS}$ in our analysis and simulations to systematically capture a wide range of squared coefficients of variation $c^2_X$:
\begin{itemize}
\item \emph{Deterministic} ($c^2_X = 0.00$): Used to model fixed service times.

\item \emph{Generalized Erlang} ($0.00 < c^2_X < 1.00$): Used to model service time distributions with less variability than the exponential distribution. We utilize a mixture of Erlang-$k$ and Erlang-$(k-1)$ distributions with the same scale parameter, fitting the parameters to match the first two moments as described in \cite[App.~B, pp.~444--446]{tijms2003first}.

\item \emph{Hyperexponential} ($c^2_X > 1.00$): Used to model service time distributions with more variability than the exponential distribution. We employ a two-phase hyperexponential distribution with balanced means, fitting the parameters to match the first two moments as described in \cite[App.~B, pp.~446--447]{tijms2003first}.
\end{itemize}
The triggering delay is always exponential, while the service times of $\CP$ and $\CS$ are drawn from the above families to achieve the desired coefficients of variation. 

We first verify the accuracy of our analytical model across a wide range of system parameters. The parameters span system loads $\rho \in \{0.30,\, 0.60,\, 0.90\}$, service time squared coefficients of variation $c^2_{X_{\CP}}=c^2_{X_{\CS}} \coloneqq c^2_{X}\in \{0.25,\, 1.00,\, 2.25\}$, mean service times $\E[X_{\CP}]=\E[X_{\CS}]=1.00$, and mean triggering delays $\E[Y] = \frac{1}{\gamma} \in \{1.00,\, 10.00\}$.  

Table~\ref{tab:comprehensive_errors} presents the absolute relative errors for the mean and variance of system times for both $\CP$ and $\CS$ customers. The relative error for a given performance metric 
$M\in\{\E[S_{\CP}], \var[S_{\CP}], \E[S_{\CS}], \var[S_{\CS}]\}$ is defined as
\[
    \frac{\left|M^{\mathrm{Analytical}}-M^{\mathrm{Simulation}}\right|}{M^{\mathrm{Simulation}}} \times 100\%.
\]

As demonstrated in Table~\ref{tab:comprehensive_errors}, the results confirm the high accuracy of the analytical model across the entire parameter space. The truncation strategy described in Section~\ref{sec:exp} (setting $\nmax$ so that $\Prob(N_{\mathrm{A}} > n_{\max}) < 10^{-9}$) introduces negligible distortion, with all relative errors remaining tightly bounded around the statistical noise of the simulation. Across all considered scenarios, the maximum observed truncation level required to satisfy this condition was $n_{\max} = 22$. This is mathematically expected: because the auxiliary queue $\mathrm{Q_A}$ operates as an $\mathrm{M/M/}\infty$ queue, its pending customer count follows a Poisson distribution whose tail decays factorially ($1/n!$). This confirms the numerical tractability of the approach, as the state space remains highly compact and fully solvable via standard linear algebra algorithms under the proposed formulation.

\subsection{Comparison with Existing Approximations}
\label{subsec:comparative_benchmarks}

To illustrate why accommodating generally distributed service times as developed in Section~\ref{sec:exp} is essential, we compare the matrix-analytic analysis of this paper against both prior analytical approximations \cite{Rahnamania2024CorrelatedArrivals} and \cite{11571174}. Both prior approaches assume a common exponential service-time distribution for the primary and secondary customers, and therefore correspond to the special case $c_X^2=1$. Importantly, the model \cite{Rahnamania2024CorrelatedArrivals} was derived exclusively for mean system times ($\E[S]$), whereas \cite{11571174} derived the LST of system times.

When the service times of both customer classes follow a common exponential distribution ($c_X^2=1.00$), the prior approximations provide reasonable estimates. However, Fig.~\ref{fig:numerical_scaling_grid} demonstrates that their accuracy deteriorates substantially when the service-time distribution deviates from the exponential case.

As shown in Fig.~\ref{fig:numerical_scaling_grid}~(a)--(d) for $\rho=0.40, 0.80$, the prior formulations~\cite{11571174,Rahnamania2024CorrelatedArrivals} cannot account for service-time variability because they are strictly restricted to the exponential case ($c_X^2=1.00$). Consequently, their predictions are depicted only at $c_X^2=1.00$. When applied to systems with general service times ($c_X^2 \neq 1.00$), these prior heuristics exhibit substantial discrepancies. In contrast, our matrix-analytic framework is accurate across $c_X^2 \in [0.10, 4.00]$.

These results demonstrate that using an exponential service-time model as an approximation for a non-exponential service-time distribution can lead to substantial inaccuracies, particularly for higher-order performance metrics such as the variance. In contrast, the proposed matrix-analytic framework accurately captures the effect of service-time variability while retaining the computational tractability of the finite-state formulation.
\section{Conclusion}
\label{sec:conclusion}

This paper introduced and analyzed the $\mathrm{M^T/G/1}$ queueing model for open-loop triggered packet streams. We established that the one-to-one causal dependence between primary and secondary packets results in a non-renewal aggregate arrival process for any non-degenerate triggering-delay distribution, precluding the direct application of conventional renewal-based queueing results.

For exponentially distributed triggering delays, the memoryless property enables a tractable two-dimensional Markovian representation based on the queue workload and the number of pending secondary packets. Exploiting this representation, we developed a finite-state matrix-analytic formulation in the Laplace--Stieltjes transform domain for general, class-dependent service-time distributions. The resulting framework characterizes the workload and the class-specific waiting- and system-time distributions, with PASTA applying to primary customers and Palm conditioning required for secondary customers.

Numerical experiments demonstrate close agreement between the analytical results and discrete-event simulations across a broad range of system loads, service-time variabilities, and triggering-delay means. In contrast to existing approximations restricted to exponential service times, the proposed framework accurately captures the impact of non-exponential service-time variability while remaining computationally tractable. Across the considered scenarios, the relative errors with respect to simulation remain very small, while the required truncation level remains modest.

The main limitation is the assumption of exponential triggering delays, which enables the finite-dimensional Markovian representation. Extending the analysis to generally distributed triggering delays would require residual-triggering-time information or suitable approximation schemes. Future work will therefore investigate stochastic bounds for system-time distributions under general triggering delays, extensions to multi-server systems, and information-freshness metrics such as the Age of Information. More broadly, the results demonstrate that explicitly accounting for causal dependencies between triggered packet transmissions is important for accurate performance evaluation of communication systems with causally dependent traffic.

\appendices

\section{Proof of Lemma~\ref{lem:model:arrival_props}}
\label{app:proof_arrival}

The lemma states three properties of the arrival process.  
Property~(i) is assumed in the text. Property~(ii) follows from the well-known fact that shifting the points of a Poisson process by i.i.d.\ random variables yields another Poisson process with the same rate~\cite{Kleinrock1975Volume1}.  
We provide here a detailed proof for property~(iii), that the aggregate arrival process is neither a Poisson nor a renewal process.

Let $N(t)$ denote the number of arrivals in the interval $(0,t]$ under \textit{stationarity}. We analyze the structure of this counting process by decomposing it into two independent components:

\begin{itemize}
    \item $N^\circ(t)$, the number of $\CS$ customers arriving in $(0,t]$ whose corresponding $\CP$ customers arrived before time $0$. The number of such pending $\CS$ customers at time $0$, say $N_0$, follows a Poisson distribution with parameter $\bar\rho \coloneqq \lambda \E[Y]$. The time until each of these customers arrives follows the residual life distribution of $Y$, denoted by $Y^{\mathrm{res}}$,  with complementary CDF
    \begin{equation*}
        \bar{F}_{Y}^{\mathrm{res}}(t) = \Prob(Y^{\mathrm{res}} > t) = \frac{1}{\E[Y]}\int_t^\infty \Prob(Y>s) \, \diff s.
    \end{equation*}
    \item $N^+(t)$, the total number of customers ($\CP$ and their corresponding $\CS$) that enter in $(0,t]$. The $\CS$ arrivals are originated from $\CP$ customers that arrive within $(0,t]$. The number of initial $\CP$ arrivals is Poisson with parameter $\lambda t$.
\end{itemize}
We first derive the Laplace--Stieltjes transform (LST) of $N^\circ(t)$. Conditioning on $N_0 = k$, each of the $k$ residual lifetimes independently expires in $(0,t]$ with probability $F_Y^{\mathrm{res}}(t)$. Hence, conditional on $N_0 = k$, $N^\circ(t)$ is binomial with parameters $(k, F_Y^{\mathrm{res}}(t))$, and
\begin{equation*}
    \E\bigl[e^{-s N^\circ(t)}\mid N_0=k\bigr]
    = \bigl(F_{Y}^{\mathrm{res}}(t)e^{-s} + \bar{F}_{Y}^{\mathrm{res}}(t)\bigr)^k .
\end{equation*}
Averaging over the Poisson($\bar\rho$) distributed $N_0$ yields
\begin{equation*}
    \E\bigl[e^{-s N^\circ(t)}\bigr]
    = \sum_{k=0}^{\infty} e^{-\bar\rho}\frac{\bar\rho^k}{k!}
      \bigl(F_{Y}^{\mathrm{res}}(t)e^{-s} + \bar{F}_{Y}^{\mathrm{res}}(t)\bigr)^k .
\end{equation*}
The sum is the Taylor expansion of an exponential, giving
\begin{equation*}
\begin{split}
        \E\bigl[e^{-s N^\circ(t)}\bigr]
    &= \exp \Bigl(\bar\rho\bigl(F_{Y}^{\mathrm{res}}(t)e^{-s} + \bar{F}_{Y}^{\mathrm{res}}(t) - 1\bigr)\Bigr)\\
    &= \exp \Bigl(\bar\rho(e^{-s}-1)F_{Y}^{\mathrm{res}}(t)\Bigr).
\end{split}
\end{equation*}

We next analyze $N^+(t)$. The $\CP$ arrivals form a Poisson process with rate $\lambda$. Consider a primary arrival at time $u\in(0,t]$. It contributes one count to $N^+(t)$, and its corresponding $\CS$ customer contributes a second count if and only if the triggering delay satisfies $Y\le t-u$. Therefore, conditional on a primary arrival at time $u$, the LST contribution is
\begin{equation*}
    e^{-s}\bigl(F_Y(t-u)e^{-s}+\bar F_Y(t-u)\bigr).
\end{equation*}
By the Poisson compounding formula,
\begin{equation*}
\begin{split}
        &\E\bigl[e^{-sN^+(t)}\bigr]
    = \\&\exp\left(
    \lambda \int_0^t
    \Bigl(
        e^{-s}\bigl(F_Y(t-u)e^{-s}+\bar F_Y(t-u)\bigr)-1
    \Bigr)\diff u
    \right).
\end{split}
\end{equation*}
After simplification and the change of variables $v=t-u$,
\begin{equation*}
\begin{split}
        &\E\bigl[e^{-sN^+(t)}\bigr]
    =\\ & \quad \quad \exp\left(
        \lambda t(e^{-s}-1)
        + \lambda e^{-s}(e^{-s}-1)\int_0^t F_Y(u)\,\diff u
    \right).
\end{split}
\end{equation*}

Since $N^\circ(t)$ and $N^+(t)$ are independent, the transform of $N(t)=N^\circ(t)+N^+(t)$ is the product of the two transforms. Using
\begin{equation*}
    \begin{split}
        \bar\rho\,F_Y^{\mathrm{res}}(t)
= \lambda\E[Y]\,F_Y^{\mathrm{res}}(t) &= \lambda\int_0^t \bar F_Y(u)\,\diff u\\
&= \lambda\Bigl(t-\int_0^t F_Y(u)\,\diff u\Bigr),
    \end{split}
\end{equation*}
we obtain
\begin{equation}
\label{eq:app:logLST_N}
\begin{split}
    \log \mathbb{E}\left[e^{-s N(t)}\right]
=& \bar\rho (e^{-s}-1)F_Y^{\mathrm{res}}(t) + \lambda t(e^{-s}-1)\\
&+ \lambda e^{-s}(e^{-s}-1)\int_0^t F_Y(u)\,\mathrm{d}u.
\end{split}
\end{equation}

We now show that $N(t)$ cannot be a stationary renewal process for any choice of $Y$. Let $\delta = e^{-s}-1$. Then \eqref{eq:app:logLST_N} can be written as
\[
\log \mathbb{E}[e^{-sN(t)}] = 2\lambda t \delta + \lambda \alpha(t)\delta^2,
\qquad
\alpha(t) := \int_0^t F_Y(u)\,\mathrm{d}u.
\]

Define
\[
\widehat{M}(s,\theta) := \int_0^\infty e^{-\theta t}\mathbb{E}[e^{-sN(t)}]\,\mathrm{d}t,
\qquad
\widehat{F}_Y(\theta) := \mathbb{E}[e^{-\theta Y}].
\]
Using $\mathcal{L}\{\alpha\} = \widehat{F}_Y(\theta)/\theta^2$ and a Taylor expansion in $\delta$, we obtain
\begin{equation}
\label{eq:agg_exp}
\begin{split}
    \widehat{M}(s,\theta)
=&\frac1\theta
+\frac{2\lambda}{\theta^{2}}\delta
+\Bigl(\lambda\frac{\widehat{F}_Y(\theta)}{\theta^{2}}+\frac{4\lambda^{2}}{\theta^{3}}\Bigr)\delta^{2}\\
&+\Bigl(-2\lambda^{2}\frac{\mathrm{d}}{\mathrm{d}\theta}\Bigl(\frac{\widehat{F}_Y(\theta)}{\theta^{2}}\Bigr)+\frac{8\lambda^{3}}{\theta^{4}}\Bigr)\delta^{3}
+O(\delta^{4}).
\end{split}
\end{equation}

If $N(t)$ were a stationary renewal process with rate $2\lambda$ and interarrival interval LST $\widehat{G}(\theta)$~\cite[Ch.~3, eq.~(6)]{Cox1962}, then
\begin{equation*}
\widehat{M}^{\mathrm{ren}}(s,\theta)
=\frac1\theta+\frac{2\lambda(e^{-s}-1)(1-\widehat{G}(\theta))}{\theta^{2}(1-e^{-s}\widehat{G}(\theta))}.
\end{equation*}
Expanding as a geometric series yields
\begin{equation}\label{eq:ren_exp}
\begin{split}
    \widehat{M}^{\mathrm{ren}}(s,\theta)
=&\frac1\theta+\frac{2\lambda}{\theta^{2}}\delta
+\frac{2\lambda\widehat{G}(\theta)}{\theta^{2}(1-\widehat{G}(\theta))}\delta^{2}\\
&+\frac{2\lambda\widehat{G}(\theta)^{2}}{\theta^{2}(1-\widehat{G}(\theta))^{2}}\delta^{3}+O(\delta^{4}).
\end{split}
\end{equation}

Equating coefficients of $\delta^2$ yields
\begin{equation}
\label{eq:rel_refined}
\frac{\widehat{G}(\theta)}{1-\widehat{G}(\theta)}
= \frac12 \widehat{F}_Y(\theta) + \frac{2\lambda}{\theta}.
\end{equation}
Matching $\delta^3$ terms and simplifying leads to the differential equation
\begin{equation}
\label{eq:ODE_refined}
\frac{\mathrm{d}}{\mathrm{d}\theta}\widehat{F}_Y(\theta)
= -\frac{1}{4\lambda}\widehat{F}_Y(\theta)^2,
\qquad \widehat{F}_Y(0)=1,
\end{equation}
whose unique solution is $\widehat{F}_Y(\theta)=\frac{4\lambda}{4\lambda+\theta}$, i.e.\ $Y \sim \mathrm{Exp}(4\lambda)$. Substituting into \eqref{eq:rel_refined} yields
\[
\widehat{G}(\theta)
= \frac{4\lambda(\theta+2\lambda)}{\theta^{2}+8\lambda\theta+8\lambda^{2}}.
\]

Insert $Y\sim\Exp(4\lambda)$ into the original (unexpanded) aggregate double transform. With $F_Y(t)=1-e^{-4\lambda t}$, we have
$\alpha(t)=t-\frac{1-e^{-4\lambda t}}{4\lambda}$.
Hence
\begin{equation*}
\begin{split}
        \widehat{M}(s,\theta)
\!=\!e^{-\delta^{2}/4}\!\!\int_{0}^{\infty}\!\!
   \exp \Bigl(\!-(\theta-2\lambda\delta-\lambda\delta^{2})t
               +\frac{\delta^{2}}{4}e^{-4\lambda t}\!\Bigr)\diff t.
\end{split}
\end{equation*}
The integral is the Laplace transform of $\exp \bigl(\frac{\delta^{2}}{4}e^{-4\lambda t}\bigr)$, shifted in $\theta$. For any non-zero $\delta$ (i.e.\ $s\neq0$), this function is \emph{not} a rational function of $\theta$: the term $\exp(c e^{-4\lambda t})$ in the time domain does not have a rational Laplace transform.

Even in this exceptional exponential case, the resulting double transform $\widehat{M}(s,\theta)$ contains a non-rational Laplace structure due to terms of the form $\exp(c e^{-4\lambda t})$, whereas any renewal transform of the form \eqref{eq:ren_exp} is necessarily rational in $\theta$. This contradiction shows that no choice of $Y$ yields a stationary renewal process representation of $N(t)$.
This establishes property~(iii) of Lemma~\ref{lem:model:arrival_props}.

% --- Bibliographic References ---
\bibliographystyle{IEEEtran}
\bibliography{IEEEexample}

\end{document}